\date{May 27, 2010. Revised September 16, 2010.}

\documentclass[a4paper,10pt,reqno]{amsart}
\usepackage{latexsym,amsmath,amsfonts,amscd,amssymb}
\usepackage{graphics}
\usepackage[all]{xy}
\usepackage{booktabs}

\newcommand\N{\mathcal{N}}

\newcommand\Q{\mathbb{Q}}
\newcommand\R{\mathbb{R}}
\newcommand\C{\mathbb{C}}
\newcommand\T{\mathbb{T}}
\renewcommand\P{\mathbb{P}}
\newcommand\F{\mathbb{F}}
\newcommand\g{\mathfrak{g}}

\newcommand\cV{\mathcal{V}}
\newcommand\cW{\mathcal{W}}

\newcommand\rk{\mathrm{rank}}
\newcommand\im{\mathrm{im}}
\newcommand\f{\varphi}

\def\l{\ell}
\newcommand{\ox}{\otimes}
\newcommand{\x}{\times}
\newcommand{\too}{\longrightarrow}
\newcommand{\bk}{\mathbf{k}}
\newcommand{\la}{\langle}
\newcommand{\ra}{\rangle}
\newcommand{\inc}{\hookrightarrow}

\newtheorem{teo}{Theorem}
\newtheorem{lemma}[teo]{Lemma}

\newtheorem{defi}[teo]{Definition}
\newtheorem{cor}[teo]{Corollary}
\newtheorem{rem}[teo]{Remark}
\theoremstyle{definition}

\newcommand\fine{\hfill$\square$\vskip 0.2 cm}

\author[G. Bazzoni]{Giovanni Bazzoni}
\address{Instituto de Ciencias Matem\'aticas
CSIC-UAM-UC3M-UCM, Consejo Superior de Investigaciones Cient\'{\i}ficas,
Serrano 113 bis, 28006 Madrid, Spain}

\email{gbazzoni@icmat.es}

\author[V. Mu\~{n}oz]{Vicente Mu\~{n}oz}
\address{Facultad de
Matem\'aticas, Universidad Complutense de Madrid, Plaza de Ciencias
3, 28040 Madrid, Spain}

\email{vicente.munoz@mat.ucm.es}

\thanks{Partially supported by Spanish grant MICINN ref.\ MTM2007-63582.}

\title[Minimal algebras over any field]{Classification of
minimal algebras over any field up to dimension $6$}

\begin{document}

\maketitle

\begin{abstract}
We give a classification of minimal algebras generated in degree $1$,
defined over any field $\bk$ of characteristic different from $2$,
 up to dimension $6$. This recovers the classification of nilpotent
 Lie algebras over $\bk$ up to dimension $6$.
  In the case of a field $\bk$ of characteristic zero, we obtain the
 classification of nilmanifolds of dimension less than or equal to
 $6$, up to $\bk$-homotopy type.
Finally, we determine which rational homotopy types of such nilmanifolds
carry a symplectic structure.
\end{abstract}

\section{Introduction and Main Results} \label{sec:intro}

Let $X$ be a nilpotent space of the homotopy type of a CW-complex of finite type over $\Q$ (all spaces considered
hereafter are of this kind). A space is nilpotent if $\pi_1(X)$ is a nilpotent group and it acts in a
nilpotent way on $\pi_k(X)$ for $k>1$.
The rationalization of $X$ (see \cite{FHT}, \cite{GM}) is a rational space $X_\Q$ (i.e. a
space whose homotopy groups are rational vector spaces) together with a map $X\to X_\Q$ inducing isomorphisms
$\pi_k(X)\ox\Q\stackrel{\cong}{\to}\pi_k(X_\Q)$ for $k\geq 1$ (recall that the rationalization of a nilpotent
group is well-defined \cite{GM}).
Two spaces $X$ and $Y$ have the same rational homotopy type if their rationalizations $X_\Q$ and $Y_\Q$ have the same
homotopy type, i.e. if there exists a map $X_\Q\to Y_\Q$ inducing isomorphisms in homotopy groups.

\medskip

The theory of minimals models developed by Sullivan \cite{S} allows to classify rational homotopy types
algebraically. In fact, Sullivan constructed a $1-1$ correspondence between nilpotent rational spaces and isomorphism classes of minimal algebras over $\Q$:
\begin{equation}\label{eq:Sullivan}
X \leftrightarrow (\wedge V_X,d)\,.
\end{equation}

Recall that, in general, a minimal algebra is a commutative differential graded algebra (CDGA henceforth) $(\wedge V,d)$ over a field $\bk$ of characteristic different from 2 in which
\begin{enumerate}
 \item $\wedge V$ denotes the free commutative algebra generated by the graded vector space $V=\oplus V^i$; 
 \item there exists a basis $\{ x_\tau,
 \tau\in I\}$, for some well ordered index set $I$, such that
 $\deg(x_\mu)\leq \deg(x_\tau)$ if $\mu < \tau$ and each $d
 x_\tau$ is expressed in terms of preceding $x_\mu$ ($\mu<\tau$).
 This implies that $dx_\tau$ does not have a linear part.
\end{enumerate}
In the above formula (\ref{eq:Sullivan}), $(\wedge V_X,d)$ is known as the minimal model of $X$. Hence, $X$ and $Y$ have the same rational homotopy type if and only if they have isomorphic minimal models (as CDGAs over $\Q$).

\medskip

The notion of \textit{real} or \textit{complex} homotopy type already appears in the literature
(cf.\ \cite{DGMS} and \cite{Morgan}): two manifolds $M_1, M_2$ have the same
real (resp.\ complex) homotopy type if the corresponding CDGAs of real (resp. complex) differential forms $(\Omega^*(M_1),d)$ and $(\Omega^*(M_2),d)$ have the same homotopy type, i.e. can be joined by a chain of morphisms inducing isomorphisms on cohomology (quasi-isomorphisms henceforth). This is equivalent to say that the two CDGAs have the same real (resp. complex) minimal model. It is convenient to remark (\cite{FHT}, \S 11(d)) that, if $(\wedge V,d)$ is the rational minimal model of $M$, then $(\wedge V\ox_\Q\R,d)$ is the real minimal model of $M$. Recall that, given a CDGA $A$ over a field $\bk$, a minimal model of $A$ is a minimal $\bk$-algebra $(\wedge V,d)$ together with a quasi-isomorphism $(\wedge V,d)\stackrel{\simeq}{\to} A$. While the minimal model of a CDGA over a field $\bk$ with $\textrm{char}(\bk)=0$ is unique up to isomorphism, the same result for arbitrary characteristic is unknown (see the appendix in which we prove uniqueness for the special case of minimal algebras treated in this paper).

\medskip

We generalize this notion to an arbitrary field $\bk$ of characteristic zero. Note that $\Q\subset \bk$.
\begin{defi}
 Let $\bk$ be a field of characteristic zero. The $\bk$-minimal model of a space $X$ is $(\wedge V_X\otimes\bk,d)$.
 We say that $X$ and $Y$ have the same
 $\bk$-homotopy type if and only if the $\bk$-minimal models $(\wedge V_X\ox \bk,d)$ and $(\wedge V_Y\ox \bk,d)$ are
 isomorphic.
\end{defi}

Note that if $\bk_1\subset \bk_2$, then the fact that $X$ and $Y$ have the same $\bk_1$-homotopy type implies that $X$ and $Y$ have the same $\bk_2$-homotopy type.

\medskip

Recall that a nilmanifold is a quotient $N=G/\Gamma$ of a nilpotent connected Lie group by a discrete co-compact
subgroup (i.e. the resulting quotient is compact). The minimal model of $N$ is precisely the Chevalley-Eilenberg complex $(\wedge\g^*,d)$ of the nilpotent Lie algebra $\g$ of $G$ (see \cite{Nomizu}). Here, $\g^*=\hom(\g,\Q)$ is assumed to be concentrated in degree 1 and the differential $d:\g^*\to \wedge^2\g^*$ reflects the Lie bracket via the pairing
 $$ 
dx(X,Y)=-x([X,Y]), \quad x\in\g^*, \ X,Y\in\g.
 $$ 
Indeed, consider a basis $\{X_i\}$ of $\g$, such that
\begin{equation}\label{eq:minimal-alg-lie}
  [X_j,X_k]= \sum_{i<j,k} a_{jk}^i \, X_i \, .
\end{equation}
Let $\{x_i\}$ be the dual basis for $\g^*$, so that
$a_{jk}^i=x_i([X_j,X_k])$. Then the differential is expressed as
\begin{equation}\label{eq:minimal-alg}
 dx_i= - \sum_{j,k>i} a_{jk}^i\, x_j x_k \, .
\end{equation}

Mal'cev proved that the existence of a basis $\{X_i\}$ of $\g$ with \emph{rational} structure constants
$a_{jk}^i$ in (\ref{eq:minimal-alg-lie}) is equivalent to the existence of a
co-compact $\Gamma \subset G$. The minimal model of the nilmanifold $N=G/\Gamma$ is
\begin{equation*}
 (\wedge(x_1,\ldots,x_n),d),
 \end{equation*}
where $V=\la x_1,\ldots x_n\ra=\oplus_{i=1}^n \Q x_i$ is the vector space generated by $x_1,\ldots, x_n$
over $\Q$, with $|x_i|=1$ for every $i=1,\ldots,n$ and $d x_i$
is defined according to (\ref{eq:minimal-alg}).

We prove the following:

\begin{teo} \label{thm:main2}
 Let $\bk$ be a field of characteristic zero. The number of minimal models of
 $6$-dimensional nilmanifolds, up to $\bk$-homotopy type, is $26+4s$, where
 $s$ denotes the cardinality of $\Q^*/((\bk^*)^2\cap \Q^*)$. In particular:
\begin{itemize}
\item There are $30$ complex homotopy types of  $6$-dimensional nilmanifolds.
\item There are $34$ real homotopy types of $6$-dimensional nilmanifolds.
\item There are infinitely many rational homotopy types of $6$-dimensional nilmanifolds.
\end{itemize}
\end{teo}

One of the consequences is the existence of pairs of nilmanifolds $M_1,M_2$ which have
the same real homotopy type, but for which there is no map $f:M_1\to M_2$ inducing
an isomorphism in the real minimal models.

Theorem \ref{thm:main2} is a consequence of the following classification of all minimal algebras generated in degree 1 by a vector space of dimension less than or equal to 6, in which we also give an explicit representative of each isomorphism class. (From now on, by the dimension of a minimal algebra $(\wedge V,d)$ we mean the dimension of $V$.)

\begin{teo} \label{thm:main}
Let $\bk$ any field of any characteristic $\mathrm{char}(\bk)\neq 2$.
There are $26+4r$ isomorphism classes of $6$-dimensional minimal algebras
generated in degree $1$ over $\bk$, where $r$ is the cardinality of $\bk^*/(\bk^*)^2$.
\end{teo}

As the Chevalley-Eilenberg complex, defined as above over a nilpotent Lie algebra, gives a one-to-one correspondence between these objects and minimal algebras generated in degree 1, we obtain the following
\begin{cor}
There are $26+4r$ isomorphism classes of $6$-dimensional nilpotent Lie algebras over $\bk$, where $r$
is the cardinality of $\bk^*/(\bk^*)^2$. In particular:
 \begin{itemize}
  \item There are $30$ isomorphism classes of $6$-dimensional nilpotent real Lie algebras.
  \item There are $34$ isomorphism classes of $6$-dimensional nilpotent complex Lie algebras.
  \item For finite fields $\bk=\F_{p^n}$, with $p\neq 2$, the cardinality of $\bk^*/(\bk^*)^2$ is $r=2$.
  So there are $34$ isomorphism classes of $6$-dimensional nilpotent Lie algebras defined
  over $\F_{p^n}$, $p\neq 2$.
  \end{itemize}
\end{cor}
This result is already known in the literature (see for instance \cite{Cerezo} or \cite{de-Graaf}),
but we obtain it from a new perspective: our starting point is the classification of
minimal models.

\medskip

Note that the classification of real 
homotopy types of $6$-dimensional
nilmanifolds already appears in the literature (see for instance \cite{Goze} and \cite{Magnin}).

\medskip

We end up the paper by determining which $6$-dimensional nilmanifolds admit a
symplectic structure. In particular, there are $27$ real homotopy types of $6$-dimensional
nilmanifolds admitting symplectic forms.
This appears already in \cite{Salamon}, but we have decided to
include it here for completeness, and to write down explicit symplectic forms in the
cases where the nilmanifold does admit them.

\medskip

\noindent \textbf{Acknowledgements.} We thank the referee for many suggestions which have
improved the presentation of the paper. We are grateful to Aniceto Murillo and Marisa Fern\'andez
for discussions on this work.

\section{Preliminaries}\label{section2}

Let $\bk$ be a field of characteristic different from $2$. Let $V=\la x_1,\ldots x_n\ra=
\oplus_{i=1}^n \bk x_i$ be a finite dimensional vector space
over $\bk$ with $\dim V\geq 2$.
We want to analyse minimal algebras of the type
 \begin{equation*}
 (\wedge(x_1,\ldots,x_n),d)
 \end{equation*}
where $|x_i|=1$, for every $i=1,\ldots,n$, and $d x_i$ is defined according to (\ref{eq:minimal-alg}), with
$a_{ij}^k\in \bk$.
Write $(\wedge V,d)$ with $V=V^1$ (i.e. $\wedge V$ is generated as an algebra by elements of degree $1$). Set
 $$
 \begin{array}{ccl}
 W_1 & = & \ker(d)\cap V\\
 W_k & = & d^{-1}(\wedge^2W_{k-1}), \ \mathrm{for} \ k\geq 2 \, .
 \end{array}
 $$
This is a filtration of $V$ intrinsically defined. We see that  $W_k\subset W_{k+1}$, for $k\geq 1$,
as follows. First notice that $W_1\subset W_2$ since $W_1=d^{-1}(0)$.
By induction, suppose that $W_{k-1}\subset W_k$; then we have
 $$
 d(W_k)=d(d^{-1}(\wedge^2W_{k-1}))\subset\wedge^2W_{k-1}\subset\wedge^2W_k\, .
 $$
This proves that $W_k\subset W_{k+1}$, as required.

Now define
 $$
 \begin{array}{ccl}
 F_1 & = & W_1\\
 F_k & = & W_k/W_{k-1} \ \mathrm{for} \ k\geq 2 \, .
 \end{array}
 $$
Then, in a non-canonical way, one has $V=\oplus F_i$.
The numbers $f_k=\dim(F_k)$ are invariants of $V$. Notice that $f_k=0$ eventually.
Under the splitting $W_k=W_{k-1}\oplus F_k$, the differential decomposes as\footnote{We use the notation
$W_{k-1}\ox F_k$ instead of $W_{k-1}\cdot F_k$, tacitly using the natural isomorphism
$W_{k-1}\cdot F_k \cong W_{k-1}\ox F_k$.
We prefer this notation, as the other one could lead to some apparent incoherences along the paper.}
 $$
 d: W_{k+1} \too \wedge^2 W_k= \wedge^2 W_{k-1} \oplus (W_{k-1}\ox F_k)\oplus \wedge^2 F_k
 $$
If we project to the second and third summands, we have
  $$
 d: W_{k+1} \too \frac{\wedge^2 W_k}{\wedge^2 W_{k-1}} = (W_{k-1}\ox F_k)\oplus \wedge^2 F_k
 $$
which vanishes on $W_k$, and hence induces a map
  \begin{equation}\label{eqn:fv}
 \bar{d}: F_{k+1} \too (W_{k-1}\ox F_k)\oplus \wedge^2 F_k =
 ((F_1\oplus\ldots\oplus F_{k-1}) \ox F_k) \oplus\wedge^2 F_k \, .
 \end{equation}
This map is injective, because $W_k=d^{-1}(\wedge^2 W_{k-1})$.
Notice that the map (\ref{eqn:fv}) is not canonical, since it depends on the choice of
the splitting.

The differential $d$ also determines a well-defined map (independent of choice of splitting)
$$
 \hat{d}: F_{k+1} \to H^2(\wedge(F_1\oplus\ldots\oplus F_k),d)\, ,
$$
which is also injective.

By considering $\bar{d}:F_2 \to \wedge^2 F_1$, we see that $f_1\geq 2$.
Moreover, if $f_1=2$ then $f_2=1$, and $\bar{d}:F_2 \to \wedge^2 F_1$ is an isomorphism.

We shall make extensive use of the following (easy) result.

 \begin{lemma} \label{lem:bilinear}
 Let $W$ be a $\bk$-vector space of dimension $k$, where $\bk$ is a field of characteristic different from $2$.
 Given any element $\f\in \wedge^2 W$, there is a (not unique) basis $x_1,\ldots, x_k$
 of $W$ such that $\f=x_1\wedge x_2 +\ldots + x_{2r-1}\wedge x_{2r}$, for some $r\geq 0$,
 $2r\leq k$.

The $2r$-dimensional space $\la x_1,\ldots, x_{2r}\ra \subset W$ is well-defined
(independent of the basis).
 \end{lemma}

 \begin{proof}
 Interpret $\f$ as
 a antisymmetric bilinear map $W^* \x W^* \to \Q$. Let $2r$ be its rank, and consider
 a basis $e_1,\ldots, e_k$ of $W^*$ such that $\f(e_{2i-1}, e_{2i})=1$, $1\leq i\leq r$,
 and the other pairings are zero. Then the dual basis $x_1,\ldots, x_k$ does the job.
 \end{proof}

\section{Classification in low dimensions}\label{sec:low-degrees}
As we said in the introduction, a
minimal algebra $(\wedge V,d)$ is of dimension $k$ if $\dim V=k$.
We start with the classification of minimal algebras over $\bk$
of dimensions $2$, $3$ and $4$.


\subsection*{Dimension $2$}
It should be $f_1=2$, so there is just one possibility:
 $$
 (\wedge(x_1,x_2),dx_1=dx_2=0)\, .
 $$
The corresponding Lie algebra is abelian.

For $\bk=\Q$, where we are classifying $2$-dimensional nilmanifolds, the corresponding
nilmanifold is the $2$-torus.

\subsection*{Dimension $3$}
Now there are two possibilities:
 \begin{itemize}
  \item $f_1=3$. Then  the minimal algebra is $(\wedge(x_1,x_2,x_3),dx_1=dx_2=dx_3=0)$.
  The corresponding Lie algebra is abelian.
  In the case $\bk=\Q$, the associated nilmanifold is the $3$-torus.
  \item $f_1=2$ and $f_2=1$. Then $\bar{d}:F_2 \to \wedge^2 F_1$ is an isomorphism. We choose a
  generator $x_3\in F_2$ such that $dx_3=x_1x_2\in \wedge^2 F_1$. The minimal algebra is
  $(\wedge(x_1,x_2,x_3),dx_1=dx_2=0, dx_3=x_1x_2)$. The corresponding Lie algebra is the
  Heisenberg Lie algebra. And for $\bk=\Q$, the associated nilmanifold is known as the
  Heisenberg nilmanifold (see \cite{Oprea-Tralle}).
  \end{itemize}

We summarize the classification in the following table:
\begin{center}
\begin{tabular}{|c|c|c|c|c|}
\hline
$(f_i)$ & $dx_1$ & $dx_2$ & $dx_3$ & $\g$\\
\hline
$(3)$ & 0 & 0 & 0 & $A_3$\\
\hline
$(2,1)$ & 0 & 0 & $x_1x_2$ & $L_3$\\
\hline
\end{tabular}
\end{center}

In the last column we have the corresponding Lie algebra: the abelian one, $A_3$, and the Lie algebra of the Heisenberg group, which we denote by $L_3$.

\subsection*{Dimension $4$}
The minimal algebra is of the form $(\wedge(x_1,x_2,x_3,x_4),d)$.
We have to consider the following cases:
 \begin{itemize}
  \item $f_1=4$. Then the $4$ elements $x_i$ have zero differential. The corresponding Lie
  algebra is abelian.
  \item $f_1=3$, $f_2=1$. As the map $\bar{d}:F_2 \to \wedge^2 F_1$ is injective, there is
  a non-zero element in the image $\f_4\in \wedge^2 F_1$. Using Lemma \ref{lem:bilinear},
  we can choose a basis $x_1,x_2,x_3$ for $F_1$ such that $\f_4=x_1 x_2$.
  Then choose $x_4\in F_2$ such that $dx_4=\f_4=x_1 x_2$. Obviously, $dx_1=dx_2=dx_3=0$.
  \item $f_1=2$, $f_2=1$, $f_3=1$. In this case, we have a basis for $F_1\oplus F_2$
  such that $dx_1=0,dx_2=0$ and $dx_3=x_1x_2$.
  The map
  $$
  \bar{d}:F_3\rightarrow F_1\ox F_2
  $$
  is injective, hence the image
  determines a line $\ell\subset F_1$ such that $\bar{d}(F_3)=\ell\ox F_2$.
  As $d(F_1\oplus F_2)=\wedge^2 F_1$, we can choose $F_3\subset W_3$ such that $d(F_3)=\ell \ox F_2$.
  We choose the basis as follows: let $x_1\in F_1$ be a vector
  spanning $\ell$; $x_2$ another vector so that $x_1,x_2$ is a basis of $F_1$;
   let $x_3\in F_2$ so that $dx_3=x_1x_2$; finally choose $x_4$ such that $dx_4=x_1x_3$.
  \end{itemize}

The results are collected in the following table:
\begin{center}
\begin{tabular}{|c|c|c|c|c|c|}
\hline
$(f_i)$ & $dx_1$ & $dx_2$ & $dx_3$ & $dx_4$ & $\g$\\
\hline
$(4)$ & $0$ & $0$ & $0$ & $0$ & $A_4$\\
\hline
$(3,1)$ & $0$ & $0$ & $0$ & $x_1x_2$ & $L_3\oplus A_1$\\
\hline
$(2,1,1)$ & $0$ & $0$ & $x_1x_2$ & $x_1x_3$ & $L_4$\\
\hline
\end{tabular}
\end{center}

The $n$-dimensional abelian Lie algebra is $A_n$; $L_4$ denotes the (unique)
irreducible $4$-dimensional nilpotent Lie algebra.

\section{Classification in dimension $5$}\label{sec:5}

The minimal algebra is of the form $(\wedge(x_1,x_2,x_3,x_4,x_5),d)$.
The possibilities for the numbers $f_k$ are the following:
$(f_1)=(5)$, $(f_1,f_2)=(4,1)$, $(f_1,f_2)=(3,2)$, $(f_1,f_2,f_3)=(3,1,1)$,
$(f_1,f_2,f_3)=(2,1,2)$, $(f_1,f_2,f_3,f_4)=(2,1,1,1)$ (noting that
$f_1\geq 2$ and that $f_1=2 \implies f_2=1$).
We study all these possibilities in detail:

\subsection*{Case $(5)$}

All the elements have zero differential. 

\subsection*{Case $(4,1)$}
Then
$F_1$ is a $4$-dimensional vector space. Now the image of $\bar{d}:F_2\to \wedge^2 F_1$ defines a line generated
by some non-zero element $\f_5\in \wedge^2 F_1$.
By Lemma \ref{lem:bilinear}, we have two cases, according to the rank of $\f_5$ (by the \textit{rank} of $\f_5$, we mean henceforth its rank as a bivector):
 \begin{enumerate}
  \item There is a basis $F_1=\langle x_1,x_2,x_3,x_4\rangle$ such that $dx_5=\f_5=x_1x_2$.
  \item There is a basis $F_1=\langle x_1,x_2,x_3,x_4\rangle$ such that $dx_5=\f_5=x_1x_2+ x_3x_4$.
 \end{enumerate}

\subsection*{Case $(3,2)$}
Now $F_1$ is a $3$-dimensional vector space,
and $\bar{d}: F_2 \hookrightarrow \wedge^2 F_1$. By Lemma \ref{lem:bilinear}, every
non-zero element $\f\in \wedge^2 F_1$ is of the form $\f=x_1x_2$ for a suitable
basis $x_1,x_2,x_3$ of $F_1$, and determines a well-defined plane $\pi=\la x_1,x_2\ra \subset F_1$.

Now $F_2 \subset \wedge^2 F_1$ is a two-dimensional vector space. Consider two linearly independent elements
of $F_2$, which give two different planes in $F_1$, and let $x_1$ be a vector spanning their intersection.
Now take a vector $x_2$ completing a basis for the first plane and a vector $x_3$ completing a basis for the second plane. Then we get the differentials $dx_4=x_1x_2$, $dx_5=x_1x_3$.

\subsection*{Case $(3,1,1)$}
$F_1$ is $3$-dimensional, and the image of $\bar{d}: F_2\hookrightarrow \wedge^2 F_1$ determines
a plane $\pi\subset F_1$. Now
 $$
 \bar{d}:F_3 \hookrightarrow F_1\ox F_2
 $$
determines a line $\ell\subset F_1$ (such that $\bar{d}(F_3)=\ell\ox F_2)$.
We easily compute
  \begin{equation}\label{enq:h2-1}
  H^2(\wedge (F_1\oplus F_2),d)= \frac{\ker(d:\wedge^2 (F_1\oplus F_2) \to \wedge^3 (F_1\oplus
   F_2))}{\im (d: F_1\oplus F_2 \to \wedge^2 (F_1\oplus F_2))} = (\wedge^2 F_1/d(F_2)) \oplus (\pi \ox F_2)\, .
  \end{equation}
(The map $d:F_1\ox F_2\inc F_1 \ox \wedge^2 F_1 \to \wedge^3 F_1$ sends $v\ox F_2\mapsto 0$ if and only
if $v\in\pi$).\linebreak Hence $\ell\subset\pi$. We can arrange a basis $x_1,x_2,x_3, x_4,x_5$ with $\ell=\la x_1 \ra$, $\pi=\la x_1,x_2\ra$, $F_1=\la x_1,x_2,x_3\ra$,
so that $\f_4=dx_4=x_1x_2$, $\f_5=dx_5=x_1 x_4+ v$, where $v\in \wedge^2 F_1$.
Recall that $F_2$, $F_3$ are not well-defined (only $W_1\subset W_2\subset W_3$ is a
well-defined filtration). In particular, this means that $\f_4$ is well-defined, but
$\f_5$ is only well defined up to $\f_5 \mapsto \f_5+ \mu \f_4$. But then
$\f_5^2 \in \wedge^4 W_2$ is well-defined, so we can distinguish cases
according to the rank (as a bilinear form) of $\f_5 \in \wedge^2(F_1\oplus F_2)$:
 \begin{enumerate}
  \item $\f_5$ is of rank $2$.
  This determines a plane $\pi'\subset W_2=F_1\oplus F_2$. The intersection of $\pi'$ with
 $F_1$ is the line $\ell$. Take an element $x_4\in \pi'$ not in the line, and declare $F_2\subset W_2$
 to be the span of $x_4$. Therefore $dx_5=x_1x_4$.
 \item $\f_5$ is of rank $4$. The vector $v$ is well-defined in $\wedge^2 F_1/d(F_2)$.
  Thus $v=a x_1x_3 + bx_2x_3$ with $b\neq 0$. We do the change of
  variables $x_4'=x_4+ax_3$, $x_3'=bx_3$. Then $x_1,x_2,x_3', x_4',x_5$ is a basis with $dx_4'=x_1x_2$,
  $dx_5'=x_1x_4'+x_2x_3'$.
 \end{enumerate}

\subsection*{Case $(2,1,2)$}
Now $F_1$ is $2$-dimensional; then $\bar{d}:F_2 \to \wedge^2 F_1$ is an isomorphism and
$\bar{d}:F_3 \to F_1\ox F_2$ is an isomorphism. Therefore
there is a basis $x_1,x_2,x_3, x_4,x_5$ such that  $dx_3=x_1x_2$, $dx_4=x_1x_3$, and $dx_5=x_2x_3$.

\subsection*{Case $(2,1,1,1)$}
Now $\bar{d}:F_2 \to \wedge^2 F_1$ is an isomorphism and
the image of $\bar{d}:F_3 \to F_1\ox F_2$ produces
a line $\ell\subset F_1$. Write $\ell=\la x_1\ra$, $F_1=\la x_1,x_2\ra$, $F_2=\la x_3\ra$ and
$F_3=\la x_4\ra$ so that $dx_3=x_1x_2$, $dx_4=x_1x_3$.

For studying $F_4$, compute
  \begin{equation}\label{eqn:h2-2}
  H^2(\wedge (F_1\oplus F_2\oplus F_3),d)= ((F_1/\ell) \ox F_2)\oplus (\ell\ox F_3 ) .
  \end{equation}
(Clearly $d(F_1\ox F_2)=0$, $d:F_1\ox F_3 \to \wedge^2 F_1 \ox F_2$ has kernel equal to $\ell\ox F_3$,
and $d:F_2\ox F_3 \to \wedge^2 F_1 \ox F_3$ is injective, so $\ker d= \wedge^2 F_1 \oplus
(F_1\ox F_2)\oplus (\ell\ox F_3 )$; on the other hand $\im \, d=\wedge^2 F_1 \oplus
(\ell \ox F_2)$.)
Recall that the element $\f_5$ generating $d(F_4)$ should have non-zero projection to $\ell\ox F_3$.
Also, $\f_5$ can be understood as a bivector in $W_3=F_1\oplus F_2\oplus F_3$.
This is well-defined up to the addition of elements in $d(W_3)=\wedge^2 F_1 \oplus
(\ell\ox F_2)$; so $\f_5^2 \in \wedge^2 W_3$ is well-defined, and hence we can talk
about the rank of $\f_5$. We have two cases:
\begin{enumerate}
 \item $\f_5$ is of rank $2$. This determines a plane $\pi'\subset W_3$, which intersects $F_1\oplus F_2$
 in a line. Let $v$ span this line and $x_4$ be another generator of $\pi'$. Write $\f_5= v x_4$. It
 must be $\la v \ra =\ell$, so $v=x_1$. Then $dx_3=x_1x_2$, $dx_4=x_1x_3$ and $dx_5=x_1x_4$.
 \item $\f_5$ is of rank $4$. Then the projection of $\f_5$ to the first summand in (\ref{eqn:h2-2})
 must be non-zero. So there is a choice of basis so
 that $dx_3=x_1x_2$, $dx_4=x_1x_3$ and $dx_5=x_1x_4+x_2x_3$.
\end{enumerate}


\subsection*{Summary of results}
We gather all the results in the following table; the first $3$
columns display the nonzero differentials.
The fourth one gives the corresponding Lie algebras,
and the last one refers to the list contained in \cite{Cerezo}:
\begin{center}
\begin{tabular}{|c|c|c|c|c|c|c|c|}
\hline
$(f_i)$ & $dx_3$ & $dx_4$ & $dx_5$ & $\g$ & \cite{Cerezo}\\
\hline
(5,0) & 0 & 0 & 0 & $A_5$ & $-$\\
\hline
(4,1)  & 0 & 0 & $x_1x_2$ & $L_3\oplus A_2$ & $-$ \\
\hline
    & 0 & 0 & $x_1x_2+x_3x_4$ & $L_{5,1}$ & $\mathcal{N}_{5,6}$\\
\hline
(3,2) & 0 & $x_1x_2$ & $x_1x_3$ & $L_{5,2}$ & $\mathcal{N}_{5,5}$\\
\hline
(3,1,1)  & 0 & $x_1x_2$ & $x_1x_4$ & $L_4\oplus A_1$ & $-$\\
\hline
    & 0 & $x_1x_2$ & $x_1x_4+x_2x_3$ & $L_{5,3}$ & $\mathcal{N}_{5,4}$\\
\hline
(2,1,2)   & $x_1x_2$ & $x_1x_3$ & $x_2x_3$ & $L_{5,5}$ & $\mathcal{N}_{5,3}$\\
\hline
(2,1,1,1)  & $x_1x_2$ & $x_1x_3$ & $x_1x_4$ & $L_{5,4}$ & $\mathcal{N}_{5,2}$\\
\hline
    & $x_1x_2$ & $x_1x_3$ & $x_1x_4+x_2x_3$ & $L_{5,6}$ & $\mathcal{N}_{5,1}$\\
\hline
\end{tabular}
\end{center}

As before, $L_{5,k}$ denote the non-split $5$-dimensional nilpotent Lie algebras.

Recall that this classification works over any field $\bk$. 
In the case $\bk=\Q$, this means in particular that there are $9$ nilpotent Lie algebras of dimension $5$
over $\Q$ and, as a consequence, $9$ rational homotopy types of $5$-dimensional nilmanifolds.

\section{Classification in dimension $6$}\label{sec:6}

Now we move to study minimal algebras of the form $(\wedge (x_1,x_2,x_3,x_4,x_5,x_6),d)$,
where $|x_i|=1$. The numbers $\{f_k\}$ can be the following:
$(f_1)=(6)$, $(f_1,f_2)=(5,1)$, $(f_1,f_2)=(4,2)$, $(f_1,f_2,f_3)=(4,1,1)$,
$(f_1,f_2)=(3,3)$,
$(f_1,f_2,f_3)=(3,2,1)$,
$(f_1,f_2,f_3)=(3,1,2)$,
$(f_1,f_2,f_3,f_4)=(3,1,1,1)$,
$(f_1,f_2,f_3,f_4)=(2,1,2,1)$,
$(f_1,f_2,f_3,f_4)=(2,1,1,2)$ and
$(f_1,f_2,f_3,f_4)=(2,1,1,1,1)$.

The case $(2,1,3)$ does not appear due to the injectivity of the differential $\bar{d}: F_3
\to W_1\ox F_2$.
Also the case $(2,1,1,2)$ does not show up, as we will see at the end of this section.
Now we consider all the cases in detail.

\subsection*{Case $(6)$}
In this case we have
$F_1=V$, $d(F_1)=0$.
This corresponds to the abelian Lie algebra.

\subsection*{Case $(5,1)$}
Here $F_1$ is a $5$-dimensional vector space and $F_2$ is $1$-dimensional, $F_2=\langle x_6\rangle$; $\bar{d}(F_2)\subset\wedge^2F_1$. Let $\f_6=dx_6\in\wedge^2F_1$ be a generator of $d(F_2)$.
By Lemma \ref{lem:bilinear}, we have the following cases:
\begin{enumerate}
 \item $\rk(\f_6)=2$. Then there exists a basis of $F_1$ such that $dx_6=x_1x_2$.
 \item $\rk(\f_6)=4$. Then there exists a basis of $F_1$ such that $dx_6=x_1x_2+x_3x_4$.
\end{enumerate}

\subsection*{Case $(4,2)$}
Here $F_1$ is a $4$-dimensional vector space and $\bar{d}:F_2\inc \wedge^2 F_1$.
This defines a projective line $\ell$ in $\P(\wedge^2 F_1)=\P^5$.

The skew-symmetric matrices of dimension $4$ with rank $\leq 2$ are given as the zero
locus of the single quadratic homogeneous equation
 $$
a_1a_6-a_2a_5+a_3a_4=0 \, ,
 $$
where
$$
A=\begin{pmatrix}
 0 & a_1 & a_2 & a_3\\
-a_1 & 0 & a_4 & a_5\\
-a_2 & -a_4 & 0 & a_6\\
-a_3 & -a_5 & -a_6 & 0\end{pmatrix}
$$
is a skew-symmetric matrix. This defines a smooth quadric $\mathcal{Q}$ in $\P^5$.

Now we have to look at the intersection of $\ell$ with $\mathcal{Q}$. Here it is where the field
of definition matters.

\begin{enumerate}
 \item $\ell\cap\mathcal{Q}=\{p_1,p_2\}$, two different points. Choose $\f_5,\f_6\in \wedge^2 F_1$
 so that they correspond to the points $p_1,p_2\in \P(\wedge^2F_1)$. Accordingly, choose
 $x_5,x_6$ generators of $F_2$ so that $\f_5=dx_5$, $\f_6=dx_6$. Note that both are bivectors
 of $F_1$ of rank $2$, but the elements $a \f_5 + b\f_6$, $ab\neq 0$ are of rank $4$.
 By Lemma \ref{lem:bilinear}, a rank $2$ element determines a plane in $F_1$. The two planes
 corresponding to $\f_5$, $\f_6$ intersect transversally (otherwise, we are in case (2) below).
 Thus we can choose a basis $x_1,x_2,x_3,x_4$ for $F_1$
 so that $dx_5=x_1x_2$ and $dx_6=x_3x_4$. Note that the elements $ax_1x_2+bx_3x_4$ are of rank $4$
 when $ab\neq 0$.

 \item $\ell\subset\mathcal{Q}$.
 We choose a basis $x_5,x_6$ so that both $\f_5=dx_5$, $\f_6=dx_6$ have rank $2$. All linear combinations
 $a dx_5+b dx_6$ are also of rank $2$. The planes determined by $\f_5, \f_6$ do not intersect
 transversally (otherwise we are in case (1) above), so they intersect in a line.
 Then we can choose a basis $x_1,x_2,x_3,x_4$ for $F_1$
 so that $dx_5=x_1x_2$ and $dx_6=x_1x_3$, the line being $\la x_1\ra$.
 Note that all elements
   $a \f_5 + b\f_6 = x_1(ax_2+bx_3)$ are of rank $2$.

 \item $\ell\cap\mathcal{Q}=\{p\}$. This means that $\ell$ is tangent to $\mathcal{Q}$. Let $\f_5
 \in\wedge^2 F_1$ corresponding to $p$. This is of rank $2$, so it determines
 a plane $\pi\subset F_1$. The plane $\pi$ is
 described by some equations $e_3=e_4=0$, where $e_3,e_4\in F_1^*$.
  Now consider $\f_6\in \wedge^2 F_1$ giving another point
  $q\in \ell$. So $\f_6$ is of rank $4$ 
  (see Lemma \ref{lem:bilinear}).
  If $\f_6(e_3,e_4)=1$, then choose $e_1,e_2$ so that $\f_6=x_1x_2+x_3x_4$, but then
  $\f_5=\lambda x_1x_2$, with $\lambda\neq 0$, and $\f_6-\lambda \f_5$ is also of rank $2$, which is a
  contradiction.\\
  Therefore $\f_6(e_3,e_4)=0$, and so $\la e_3,e_4\ra$ is Lagrangian in $(F_1^*,\f_6)$.
  We can complete the basis to $e_1,e_2,e_3,e_4$ so that
  $dx_6=\f_6=x_1x_3+x_2x_4$. Normalize $\f_5$ so that $dx_5=\f_5=x_1x_2$.
  All forms $dx_6+a \, dx_5$ are of rank $4$.

\item $\ell\cap\mathcal{Q}=\emptyset$.
This means that $\ell$ and $\mathcal{Q}$ intersect in two points with coordinates in the algebraic closure of
$\bk$. As this intersection is invariant by the Galois group, there must be a quadratic extension
$\bk'\supset \bk$ where the coordinates of the two points lie; the two points are conjugate by the Galois
automorphism of $\bk'|\bk$. Therefore, there is an element $a\in\bk^*$
such that $\bk'=\bk(\sqrt{a})$, $a$ is not a square in $\bk$, and the differentials
 $$
 dx_5=x_1x_2, \qquad dx_6=x_3x_4.
 $$
satisfy that the planes $\pi_1=\la x_1,x_2 \ra$ and $\pi_2=\la x_3,x_4 \ra$ are conjugate under the
Galois map $\sqrt{a}\mapsto -\sqrt{a}$. Write:
 \begin{eqnarray*}
  x_1 &=& y_1 + \sqrt{a} y_2, \\
  x_2 &=& y_3 + \sqrt{a} y_4, \\
  x_3 &=& y_1 - \sqrt{a} y_2, \\
  x_4 &=& y_3 - \sqrt{a} y_4, \\
  x_5 &=& y_5 + \sqrt{a} y_6, \\
  x_6 &=& y_5 - \sqrt{a} y_6,
 \end{eqnarray*}
where $y_1,\ldots, y_6$ are defined over $\bk$.
Then $dy_5=y_1y_3 + ay_2y_4$, $dy_6=y_1y_4 + y_2y_3$.

This is the ``canonical'' model. Two of these minimal algebras are not isomorphic over $\bk$ 
for different quadratic field extensions, since
the equivalence would be given by a $\bk$-isomorphism, therefore commuting with
the action of the Galois group.

The quadratic field extensions are parametrized by elements $a\in \bk^*/(\bk^*)^2- \{1\}$. Note that
for $a=1$, we recover case (1), where $dy_5+dy_6=(y_1+y_2)(y_3+y_4)$ and $dy_5-dy_6=(y_1+y_2)(y_3-y_4)$
are of rank $2$.

\end{enumerate}

\begin{rem}
If $\bk=\C$ (or any algebraically closed field) then case (4) does not appear.

For $\bk=\R$, we have that $\R^*/(\R^*)^2- \{1\}=\{-1\}$, and there is only one
minimal algebra in this case, given by
$dy_5=y_1y_3 - y_2y_4$, $dy_6=y_1y_4 + y_2y_3$.

The case $\bk=\Q$ is very relevant, as it corresponds to the classification of
rational homotopy types of nilmanifolds. Note that in this case the classes in
$\Q^*/(\Q^*)^2$ are parametrized bijectively by elements $\pm p_1p_2\ldots p_k$, where
$p_i$ are different primes, and $k\geq 0$.
In particular, if $a$ is a square in $\Q$ then we fall again in (1) above.
\end{rem}

\begin{rem}
Note that we get examples of distinct rational homotopy types of nilmanifolds which have the same
real homotopy type. Also, we get nilmanifolds with different real homotopy types
but the same complex homotopy type.
\end{rem}

\subsection*{Case $(4,1,1)$}

Now $F_1$ is $4$-dimensional, and $\bar{d}:F_2\inc \wedge^2 F_1$ determines an element $\f_5\in \wedge^2 F_1$.
Clearly, $\wedge^2(F_1\oplus F_2)=\wedge^2 F_1 \oplus (F_1 \ox F_2)$.
The differential $d:F_1\ox F_2\to \wedge^3F_1$ is given as wedge by $\f_5$. So if $\f_5$ is of rank $4$,
then this map is an isomorphism and
  $$
 \ker (d:\wedge^2(F_1\oplus F_2) \to \wedge^3(F_1\oplus F_2) )=\wedge^2 F_1.
  $$
So there cannot be an injective map
$\bar{d}:F_3 \to F_1\ox F_2$.
This shows that $\f_5$ must be of rank $2$, and therefore it determines  a plane
$\pi\subset F_1$. Now the closed elements are given as $\wedge^2F_1 \oplus (\pi\ox F_2)$. The differential
$\bar{d}:F_3\to \pi\ox F_2$ determines a line $\ell\subset\pi$. Let $x_1$ be a generator for $\ell$, and
$\pi=\la x_1,x_2\ra$. Then there is a basis $x_1,x_2,x_3,x_4$ such that $dx_5=x_1x_2$ and $dx_6= x_1 x_5 +
\f'$, where
$\f'\in \wedge^2 F_1$. We are allowed to change $x_5$ by $x_5'=x_5+v$ with $v\in F_1$. This has
the effect of changing $dx_6$ by adding $x_1 v$. This means that we may assume that $\f'$ does
not contain $x_1$, so $\f'\in \wedge^2 (F_1/\ell)$.
Actually, wedging $\f_6=dx_6\in \wedge^2 F_1 \oplus (\pi\ox F_2)$ by $x_1$, we get an element
$\f_6\, x_1 \in \wedge^3 F_1$ which is the image of $\f'$ under the map $\wedge^2 (F_1/\ell)
\stackrel{x_1}{\inc} \wedge^3 F_1$. It is then easy to see then that $\f'$ is well-defined
(independent of the choices of $F_2$, $F_3$).

We have the following cases:
  \begin{enumerate}
  \item $\f'=0$. So $dx_6= x_1 x_5$.
  \item $\f'$ is non-zero, so it is of rank $2$. Therefore it determines a plane $\pi'$ in $F_1/\ell$.
  If this is transversal to the line $\pi/\ell$, then $\f'=x_3x_4$ and we have that $dx_6=x_1x_5+x_3x_4$.
  \item If $\pi'$ contains $\pi/\ell$, then $\f'=x_2x_3$ and
  we have $dx_6=x_1x_5+x_2x_3$. 
  \end{enumerate}

\subsection*{Case $(3,3)$}
This case is very easy, since $F_1$ is three-dimensional, and $\bar{d}:F_2\to \wedge^2F_1$ must be an isomorphism.
So there exists a basis such that $dx_4=x_1x_2$, $dx_5=x_1x_3$ and $dx_6=x_2x_3$.

\subsection*{Case $(3,2,1)$}
We have a three-dimensional space $F_1$. Then there is a two-dimensional space $F_2$ with a map $\bar{d}:F_2\to
\wedge^2 F_1$. Note that any element in $F_2$ determines a plane in $F_1$. Intersecting those planes, we get a
line
$\ell\subset F_1$. Then the differential gives an isomorphism $h:F_2 \stackrel{\cong}{\to} F_1/\ell$
(defined up to a non-zero scalar).
Choosing $\ell=\la x_1 \ra$, we take basis such that $h(x_4)=x_2$ and $h(x_5)=x_3$. So
 $$
 dx_4=x_1x_2, \quad dx_5=x_1x_3\,.
 $$

Let us compute the closed elements in $\wedge^2 (F_1\oplus F_2)=
\wedge^2 F_1 \oplus (F_1\ox F_2) \oplus \wedge^2 F_2$.
Clearly, $d: \wedge^2F_2 \inc \wedge^2 F_1 \ox F_2$. Also
the map $d:F_1\ox F_2 \cong F_1\ox (F_1/\ell)\to \wedge^3 F_1$
is the map $(u,v)\mapsto u\wedge v\wedge x_1$. As $\im \, d=d(F_2)$, we have
that
  $$
  H^2(\wedge (F_1\oplus F_2), d)=  \wedge^2(F_1/\ell) \oplus \ker (F_1\ox F_2 \to \wedge^3 F_1),
  $$
and $F_3$ determines an element $\f_6$ in that space.
Let $\pi_4$, $\pi_5$ be the planes in $F_1$ corresponding to $dx_4$, $dx_5$. There are
vectors $v_2\in \pi_4$, $v_3\in \pi_5$ and $\lambda\in \bk$
so that $\f_6= \lambda x_2x_3+ v_2x_4+v_3x_5$. We have
the following cases:

\begin{enumerate}
\item Suppose that $\f_6^2 x_1 \neq 0$ (this condition is well-defined, independently of the choices of
$F_2$, $F_3$). This is an element in $\wedge^3 F_1 \ox \wedge^2 F_2 \cong x_1 \ox \wedge^2(F_1/\ell) \ox
\wedge^2 F_2 \cong  (\wedge^2 F_2)^2$. Taking an isomorphism
$\wedge^2 F_2\cong \bk$, we have that the class of $\f_6^2 x_1 \in (\wedge^2 F_2)^2\cong \bk$
gives a well-defined element in $\bk^*/(\bk^*)^2$.\\
The condition $\f_6^2 x_1 \neq 0$ translates into
$v_2,v_3,x_1$ being linearly independent. So we can arrange $x_2=a_2v_2$, $x_3=a_3v_3$, with $a_2,a_3\neq 0$.
Normalizing $x_6$, we can assume $a_2=1$. So
$dx_4=x_1x_2$, $dx_5=x_1x_3$, $dx_6=\lambda x_2x_3+x_2x_4 +a  x_3x_5$.
Note that the class defined by $\f_6^2 x_1$ is $-2a\in \bk^*/(\bk^*)^2$.
(If we change the basis $x_3'=\mu x_3$, $x_5'=\mu x_5$ we obtain $dx_6=x_2x_4 +a\mu^{-2} x_3'x_5'$.
We see again that $-2a$ is defined in $\bk^*/(\bk^*)^2$).\\
Changing the basis as $x'_4=x_4+\lambda x_3$, we get
$dx'_4=x_1x_2$, $dx_5=x_1x_3$, $dx_6=x_2x'_4 -\frac{a}{2} x_3x_5$.

\item Now suppose $\f_6^2 x_1 = 0$, $\f_6 x_1\not\in \wedge^3 F_1$
and $\f_6^2\not\in \wedge^3 F_1\ox F_2$ (again these conditions are independent of the choices of
$F_2$, $F_3$). Then $v_2v_3x_1=0$ and $v_2v_3\neq 0$.
We can choose the coordinates $x_2,x_3$ (and $x_4,x_5$
accordingly through $h$) so that $v_2=x_2$, $v_3=x_1$. Therefore
$\f_6= \lambda x_2x_3+ x_2x_4+x_1x_5$. Now the change of variable $x_4'=x_4+\lambda x_3$ gives the form
$dx_4'=x_1x_2$, $dx_5=x_1x_3$, $dx_6=x_2x_4'+x_1x_5$.

\item Suppose that $\f_6^2 \in \wedge^3F_1\ox F_2$ and $\f_6 x_1\not\in \wedge^3 F_1$.
Then $v_2v_3=0$ but $x_1$ is linearly independent with $\la v_2,v_3\ra$. Choose coordinates
so that $v_2=x_2$ and $v_3=0$. So $\f_6= \lambda x_2x_3+ x_2x_4$. The change of
variable $x_4'=x_4+\lambda x_3$ gives the form
$dx_4'=x_1x_2$, $dx_5=x_1x_3$, $dx_6=x_2x_4'$.

\item Suppose that $\f_6x_1\in \wedge^3 F_1$, $\f_6^2\neq 0$. So that we can choose $v_2=x_1$, $v_3=0$. We
have $dx_4=x_1x_2$, $dx_5=x_1x_3$, $dx_6=\lambda x_2x_3 + x_1x_4$, where $\lambda\neq 0$.
Now take $x_3'=\lambda x_3$ and
$x_5'=\lambda x_5$. So $dx_4=x_1x_2$, $dx'_5=x_1x'_3$, $dx_6=x_2x'_3 + x_1x_4$

\item Finally, we have $\f_6x_1\in \wedge^3 F_1$, $\f_6^2= 0$ and this gives the minimal algebra
$dx_4=x_1x_2$, $dx_5=x_1x_3$, $dx_6= x_1x_4$.
\end{enumerate}

\subsection*{Case $(3,1,2)$}
We have a $3$-dimensional vector space $F_1$. Then $\bar{d}:F_2\to \wedge^2F_1$ determines a well-defined plane
$\pi\subset F_1$. Looking at $\wedge^2(F_1\oplus F_2)=\wedge^2 F_1 \oplus (F_1\ox F_2)$, we see that the
closed elements are $\wedge^2 F_1 \oplus (\pi \ox F_2)$. The differential is defined by
 \begin{equation}\label{eqn:312}
  \hat{d}:F_3 \to H^2(\wedge(F_1\oplus F_2), d)= (\wedge^2 F_1/ d(F_2)) \oplus (\pi\ox F_2)\, ,
 \end{equation}
where the projection $\bar{d}: F_3\to \pi\ox F_2$ is injective, hence an isomorphism.
So we identify $F_3\cong \pi\ox F_2$.
Let $x_1,x_2$ be a basis for $\pi$, and $x_5,x_6$ the corresponding basis of $F_3$ through the above
isomorphism. So $dx_4=x_1x_2$, $dx_5=x_1x_4 + v_5$, $dx_6=x_2x_4+v_6$, where $v_5,v_6\in
\wedge^2 F_1/ d(F_2)$.

The map (\ref{eqn:312}) together with $\bar{d}^{-1}:\pi\ox F_2 \to F_3$
gives a map $\phi:\pi\ox F_2 \to (\wedge^2 F_1/ d(F_2))$. It is easy to see that the pairing
$F_1\ox \wedge^2 F_1\to \wedge^3 F_1$ induces a non-degenerate pairing
$\pi \ox (\wedge^2 F_1/ d(F_2))\to \wedge^3 F_1$,
and hence an isomorphism $(\wedge^2 F_1/ d(F_2)) \cong \pi^*\ox \wedge^3 F_1$. Hence
$\phi :\pi\ox F_2 \to\pi^*\ox \wedge^3 F_1$, and using that $\pi^* \cong\pi \ox \wedge^2 \pi^*$,
we finally get a map
 $$
 \phi: \pi \to \pi\ox (\wedge^2 \pi^*\ox \wedge^3 F_1\ox F_2^*).
 $$
This gives an endomorphism of $\pi$ defined up to a constant.

Now let us see the indeterminacy of $\phi$. With the change of variables
$x'_4 = x_4+\mu x_3 + \nu x_2 + \eta x_1$ we get $dx_5=x_1x'_4 + v'_5$,
$dx_6=x_2x'_4 + v'_6$, where
$v'_5=v_5-\mu x_1x_3$, $v'_6=v_6-\mu x_2x_3$. Therefore the corresponding map $\phi'=\phi-\mu \, \mathrm{Id}$.
So $\phi$ is defined up to addition of a multiple of the identity.

We get the following classification:
 \begin{enumerate}
  \item Suppose that $\phi$ is zero (or a scalar multiple of the identity). Then $dx_4= x_1x_2$, $dx_5=x_1x_4$,
  $dx_6= x_2x_4$.
  \item Suppose that $\phi$ is diagonalizable. Adding a multiple of the identity, we can assume that
  one of the eigenvalues is zero and the other is not. Let $x_2$ generate the image and $x_1$ be
  in the kernel. Then $dx_4= x_1x_2$, $dx_5=x_1x_4$, $dx_6= x_2x_4 + x_2x_3$.
  \item Suppose that $\phi$ is not diagonalizable. Adding a multiple of the identity, we can assume
  that the eigenvalues are zero. Let $x_1$ generate the image, so that
      $x_1$ is in the kernel. Then $dx_4= x_1x_2$, $dx_5=x_1x_4$, $dx_6= x_2x_4+x_1x_3$.
  \item Finally, $\phi$ can be non-diagonalizable if $\bk$ is not algebraically closed.
  To diagonalize $\phi$ we need a quadratic extension of $\bk$. Let $a\in \bk^*$ so that
  $\phi$ diagonalizes over $\bk'=\bk(\sqrt{a})$.
 If we arrange $\phi$ to have zero trace (by adding a multiple of the identity), then
 the minimum polynomial of $\phi$ is $T^2-a$. So
 we can choose a basis such that $\phi(x_1)=x_2$, $\phi(x_2)= a x_1$.
 Thus  $dx_4= x_1x_2$, $dx_5= x_1x_4 + x_2x_3$,
 $dx_6= x_2x_4 + a x_1x_3$. The minimal algebras are parametrized by $a\in \bk^*/(\bk^*)^2-\{1\}$.
 (The value $a=1$ recovers case (2)).
 \end{enumerate}

\subsection*{Case $(3,1,1,1)$}
Now $F_1$ is of dimension $3$. We have a one-dimensional space given as the image of
$\bar{d}:F_2\inc \wedge^2 F_1$,
which determines a plane $\pi\subset F_1$.
The closed elements in $\wedge^2 (F_1\oplus F_2)$
are $\wedge^2F_1 \oplus (\pi\ox F_2)$. Therefore, $\f_5=dx_5$ determines a line
$\ell\subset \pi$. But it also determines an element in $\wedge^2 F_1$, up to $d(F_2)$ and
up to $\ell\wedge F_1$, i.e. in $\wedge^2 (F_1/\ell)$. Then
 \begin{enumerate}
  \item $dx_4=x_1x_2$, $dx_5=x_1x_4$. Now we compute the closed elements in
  $\wedge^2(F_1\oplus F_2 \oplus F_3)$ to be $\wedge^2F_1 \oplus (\pi\ox F_2)\oplus(\l\ox F_3)$.
  The element $\f_6=dx_6$ has non-zero last component in $\ell\ox F_3$. It is well-defined
   up to $\ell\wedge F_1$ and up to $\ell\ox F_2$. There are several cases:
   \begin{enumerate}
   \item $\f_6\in \ell\ox F_3$. Then $dx_6=x_1x_5$.
   \item $\f_6 \in (\pi\ox F_2) \oplus  (\ell\ox F_3)$. Then $dx_6=x_2 x_4 +x_1x_5$.
   \item $\f_6 \in \wedge^2 F_1 \oplus  (\ell\ox F_3)$, then $dx_6=x_2x_3+x_1x_5$.
   \item $\f_6$ has non-zero components in all summands. Then $dx_6=\lambda x_2x_3 + x_2x_4+x_1x_5$. We
   can arrange $\lambda=1$ by choosing $x_3'=\lambda x_3$.
 \end{enumerate}
(We can check that these cases are not equivalent: the first one is characterised by $\f_6 x_1=0$;
the second one by $\f_6 x_1\neq 0$, $\f_6\f_5=0$; the third one by $\f_6 x_1\neq0$, $\f_6\f_5\neq 0$,
$\f_6\f_4=0$; the last one by $\f_6 x_1\neq 0$, $\f_6\f_5\neq 0$, $\f_6\f_4\neq 0$).

  \item $dx_4=x_1x_2$, $dx_5=x_1x_4+x_2x_3$. Then the closed elements  in
  $\wedge^2(F_1\oplus F_2 \oplus F_3)$ are those in
   $$
   \wedge^2F_1 \oplus (\pi\ox F_2)\oplus  \la x_1x_5+x_4x_3\ra .
   $$
   So $\f_6= ax_1x_3 +b x_2x_3+ c x_1x_4+ d x_2x_4 + x_1x_5+x_4x_3$. The change of
   variables $x_6'=x_6-bx_5$ arranges $b=0$. Then the change of variables
     $x'_3=-dx_2+x_3$ and $x'_5=ax_3+x_5$ arranges $a=0$ and $d=0$. Thus $\f_6= c x_1x_4 + x_1x_5+x_4x_3$.
   Finally $x_3'=-\frac{c}{2}x_1+x_3$, $x_5'=\frac{c}{2}x_4+x_5$ arranges $c=0$. Hence $\f_6=x_1x_5-x_3x_4$.

\end{enumerate}

\subsection*{Case $(2,1,2,1)$}
Now we have a $2$-dimensional space $F_1$, and an isomorphism $\bar{d}:F_2\to \wedge^2 F_1$.
Also $\bar{d}:F_3 \to \wedge^2(F_1\oplus F_2)/\wedge^2 F_1 =F_1\ox F_2$ is an isomorphism. Then
there is a basis for $F_1\oplus F_2\oplus F_3$ such that
 $$
 dx_3=x_1x_2, \ dx_4=x_1x_3 \ \ \mathrm{and} \ dx_5=x_2x_3\, .
 $$
Let us compute the closed elements in $\wedge^2(F_1\oplus F_2\oplus F_3)$.
First, $d:F_2\ox F_3  \to \wedge^2 F_1 \ox F_3$ is an isomorphism; second
$d:\wedge^2 F_3\inc F_1\ox F_2 \ox F_3$ is an injection; finally,
$d: F_1\ox F_3 \cong F_1 \ox F_1\ox F_2 \to \wedge^2 F_1\ox F_2$. So the kernel of $d$
is isomorphic to $\wedge^2(F_1\oplus F_2)\oplus (s^2 F_1\ox F_2)$.
Then
  $$
  \f_6\in H^2(\wedge(F_1\oplus F_2\oplus F_3),d)= s^2 F_1 \subset F_1\ox F_1\cong F_1\ox F_3
  $$
determines a non-zero quadratic form on $F_1$ up to multiplication by scalar, call it $A$.
(Here we use the natural identification $F_3\cong F_1$, $x_4\mapsto x_1$, $x_5\mapsto x_2$,
defined up to scalar).

We have the following cases:
\begin{enumerate}
 \item If $\rk(A)=1$, then $A$ has non-zero kernel. We get a basis such that
 $dx_6=x_1x_4$.
 \item If $\rk(A)=2$ then $\det(A)\neq 0$.
This determines a $2\x 2$-matrix $A$ defined up to conjugation $A\mapsto M^T A M$ and
up to $A\mapsto \lambda A$. Note that the class of the determinant $a=\det(A)\in \bk^*/(\bk^*)^2$
is well-defined. Take a basis diagonalizing $A$. We can arrange that
$A=\begin{pmatrix} 1 & 0 \\ 0 & a \end{pmatrix}$.
So $dx_6=x_1x_4+ a x_2x_5$. 
(Note that for $a=0$ we recover case (1)).
\end{enumerate}

\subsection*{Case $(2,1,1,2)$}
Now $F_1$ is $2$-dimensional, and $\bar{d}:F_2 \to \wedge^2 F_1$ is an isomorphism.
$F_3$ is one-dimensional and $\bar{d}: F_3 \to \wedge^2 (F_1\oplus F_2) /\wedge^2 F_1=F_1\ox F_2$.
Therefore there exists a line $\ell\subset F_1$ such that $d(F_3)=\ell\ox F_2$.

We compute the closed elements in $\wedge^2(F_1\oplus F_2\oplus F_3) = \wedge^2 F_1 \oplus (F_1\ox F_2) \oplus
(F_1\ox F_3) \oplus (F_2\ox F_3)$. As $d: F_1\ox F_3 \to \wedge^2 F_1\ox F_2$ has kernel $\ell\ox F_3$ and
$d:F_2\ox F_3\inc \wedge^2 F_1\ox F_3$, we have that
 $$
 H^2(\wedge(F_1\oplus F_2\oplus F_3),d)= ((F_1/\ell)\ox F_2)\oplus (\ell\ox F_3)\, .
 $$
As $\bar{d}:F_4 \to \wedge^2(F_1\oplus F_2\oplus F_3)/\wedge^2 (F_1\oplus F_2)$
is injective, and $\dim(\ell\ox F_3)=1$, it cannot be that $f_4=2$.

\subsection*{Case $(2,1,1,1,1)$}
We work as in the previous case. Now $\bar{d}:F_4 \to ((F_1/\ell)\ox F_2) \oplus (\ell \ox F_3)$ produces
an isomorphism $F_4 \cong \ell\ox F_3$ and hence a map
 $$
 \phi: \ell\ox F_3 \to (F_1/\ell)\ox F_2.
 $$
Note that this map is well-defined, independent of the choice of $F_3$ satisfying $W_2\oplus F_3=W_3$.
We have the following cases
\begin{enumerate}
\item Suppose that $\phi=0$. So there is a basis such that
$dx_3 = x_1x_2$, $dx_4 = x_1x_3$, $dx_5 = x_1x_4$, where we have chosen
$\ell=\la x_1 \ra$, $F_1=\la x_1,x_2\ra$.
We can easily compute
 $$
 H^2(\wedge(x_1,x_2,x_3,x_4,x_5),d)=\la x_1x_5, x_2x_3, x_2x_5-x_3x_4 \ra\, .
 $$
Then
 \begin{equation}\label{eqn:f6}
  \f_6= dx_6 = a x_1x_5  + b x_2x_3+ c(x_2x_5-x_3x_4 )\, .
 \end{equation}
We have
 \begin{enumerate}
  \item If $\f_6 x_1=0$ then $b=c=0$. We can choose generators so that $dx_6=x_1x_5$.
  \item If $\f_6x_1\neq 0$ and $\f_6 x_1x_2=0$, then $c=0$ and $a,b\neq 0$. We can
  arrange $a=1$ by normalizing $x_6$ and then do the change of variables $x_2'=bx_2$
  $x_3'=bx_3$, $x_4'=bx_4$, $x_5'=bx_5$, $x_6'=bx_6$. This produces
  an equation as (\ref{eqn:f6}) with $b=1$. Hence $dx_6 = x_1x_5 + x_2x_3$.
  \item If $\f_6 x_1x_2 \neq 0$, then $c\neq 0$. We can arrange $c=1$ by normalizing $x_6$.
   Now put $x_2'=x_2+ax_1$ to arrange $a=0$. Finally take $x_5'=x_5+bx_3$, $x_4'=x_4+bx_2$ to
   be able to put $b=0$. So $dx_6=x_2x_5-x_3x_4$.
  \end{enumerate}
\item Suppose that $\phi\neq 0$. Then there is a basis for $F_1\oplus F_2\oplus F_3\oplus F_4$
 such that $dx_3 = x_1x_2$, $dx_4 = x_1x_3$, $dx_5 = x_1x_4+x_2x_3$.
 We can easily compute
 $$
 H^2(\wedge(x_1,x_2,x_3,x_4,x_5),d)=\la x_1x_4, x_1x_5+x_2x_4, x_2x_5-x_3x_4 \ra\, .
 $$
Then
 $$
  \f_6= dx_6 = a x_1x_4  + b (x_1x_5+x_2x_4) + c(x_2x_5-x_3x_4 )\, .
 $$
We have
 \begin{enumerate}
  \item If $\f_6 x_1x_2=0$ then $c=0$.  We can suppose $b=1$, and put $x_2'=x_2+\frac{a}{2}x_1$, $x_5'=x_5+\frac{a}{2}x_4$,
  to arrange $a=0$. So $dx_6=x_1x_5+x_2x_4$.
  \item If $\f_6x_1x_2\neq 0$ then we can suppose $c=1$.
  Put $x_2'=bx_1+x_2$ and $x_5'=bx_4+x_5$ to eliminate $b$. Finally do the change of
  variables $x_4'=x_4-\frac{a}{2}x_2$, $x_5'=x_5-\frac{a}{2}x_3$ and $x_6'=-ax_5+x_6$ to arrange $a=0$.
  Hence $dx_6=x_2x_5-x_3x_4$.
  \end{enumerate}

\end{enumerate}

\subsection*{Classification of minimal algebras over $\bk$}
Let $\bk$ be any field of characteristic different from $2$. The above work can be summarized in Table 1.

\begin{table}[h]\label{table:1}
\caption{Classification of minimal algebras over $\bk$}
\begin{tabular}{|c|c|c|c|c|c|c|}
\hline
 $(f_i)$    & $dx_3$ & $dx_4$ & $dx_5$ & $dx_6$ & $\g$ \\
\hline
(6,0)    & 0 & 0 & 0 & 0 & $A_6$\\
\hline
(5,1)   & 0 & 0 & 0 & $x_1x_2$ & $L_3\oplus A_3$\\
\hline
     & 0 & 0 & 0 & $x_1x_2+x_3x_4$ & $L_{5,1}\oplus A_1$\\
\hline
(4,2)  & 0 & 0 & $x_1x_2$ & $x_1x_3$ & $L_{5,2}\oplus A_1$\\
\hline
     & 0 & 0 & $x_1x_2$ & $x_3x_4$ & $L_3\oplus L_3$\\
\hline
     & 0 & 0 & $x_1x_2$ & $x_1x_3+x_2x_4$ & $L_{6,1}$\\
\hline
     & 0 & 0 & $x_1x_3+ax_2x_4$ & $x_1x_4+x_2x_3$ & $L_{6,2}^a, a\in\Lambda-\{1\}$\\
\hline
(4,1,1)   & 0 & 0 & $x_1x_2$ & $x_1x_5$ & $L_4\oplus A_2$\\
\hline
     & 0 & 0 & $x_1x_2$ & $x_1x_5+x_3x_4$ & $L_{6,3}$\\
\hline
     & 0 & 0 & $x_1x_2$ & $x_1x_5+x_2x_3$ & $L_{5,3}\oplus A_1$\\
\hline
(3,3)    & 0 & $x_1x_2$ & $x_1x_3$ & $x_2x_3$ & $L_{6,4}$\\
\hline
(3,2,1)   & 0 & $x_1x_2$ & $x_1x_3$ & $x_1x_4$ & $L_{6,5}$\\
\hline
     & 0 & $x_1x_2$ & $x_1x_3$ & $x_2x_4$ & $L_{6,6}$\\
\hline
     & 0 & $x_1x_2$ & $x_1x_3$ & $x_1x_5+x_2x_4$ & $L_{6,7}$\\
\hline
     & 0 & $x_1x_2$ & $x_1x_3$ & $x_2x_4+ax_3x_5$ & $L_{6,8}^a, a\in\Lambda$\\
\hline
     & 0 & $x_1x_2$ & $x_1x_3$ & $x_1x_4+x_2x_3$ & $L_{6,9}$\\
\hline
(3,1,2)  & 0 & $x_1x_2$ & $x_1x_4$ & $x_2x_4$ & $L_{5,5}\oplus A_1$\\
\hline
     & 0 & $x_1x_2$ & $x_1x_4$ & $x_2x_3+x_2x_4$ & $L_{6,10}$\\
\hline
     & 0 & $x_1x_2$ & $x_1x_4$ & $x_1x_3+x_2x_4$ & $L_{6,11}$\\
\hline
     & 0 & $x_1x_2$ & $x_1x_4+x_2x_3$ & $x_1x_3+ax_2x_4$ & $L_{6,12}^a, a\in\Lambda-\{1\}$\\
\hline
(3,1,1,1) & 0 & $x_1x_2$ & $x_1x_4$ & $x_1x_5$ & $L_{5,4}\oplus A_1$\\
\hline
     & 0 & $x_1x_2$ & $x_1x_4$ & $x_1x_5+x_2x_3$ & $L_{6,13}$\\
\hline
     & 0 & $x_1x_2$ & $x_1x_4$ & $x_1x_5+x_2x_4$ & $L_{5,6}\oplus A_1$\\
\hline
     & 0 & $x_1x_2$ & $x_1x_4$ & $x_1x_5+x_2x_3+x_2x_4$ & $L_{6,14}$\\
\hline
     & 0 & $x_1x_2$ & $x_1x_4+x_2x_3$ & $x_1x_5-x_3x_4$ & $L_{6,15}$\\
\hline
(2,1,2,1)   & $x_1x_2$ & $x_1x_3$ & $x_2x_3$ & $x_1x_4$ & $L_{6,16}$\\
\hline
     & $x_1x_2$ & $x_1x_3$ & $x_2x_3$ & $x_1x_4+ax_2x_5$ & $L_{6,17}^a, a\in\Lambda$\\
\hline
(2,1,1,1,1)    & $x_1x_2$ & $x_1x_3$ & $x_1x_4$ & $x_1x_5$ & $L_{6,18}$\\
\hline
     & $x_1x_2$ & $x_1x_3$ & $x_1x_4$ & $x_1x_5+x_2x_3$ & $L_{6,19}$\\
\hline
     & $x_1x_2$ & $x_1x_3$ & $x_1x_4$ & $x_2x_5-x_3x_4$ & $L_{6,20}$\\
\hline
     & $x_1x_2$ & $x_1x_3$ & $x_1x_4+x_2x_3$ & $x_1x_5+x_2x_4$ & $L_{6,21}$\\
\hline
     & $x_1x_2$ & $x_1x_3$ & $x_1x_4+x_2x_3$ & $x_2x_5-x_3x_4$ & $L_{6,22}$\\
\hline
\end{tabular}
\end{table}

The first $4$ columns display the non-zero differentials, and the fifth one is a labelling
of the corresponding Lie algebra. Denote $\Lambda=\bk^* /(\bk^*)^2$. There are $4$ families which are indexed by a parameter $a$: $L_{6,2}^a$ and $L_{6,12}^a$, which are indexed by $a\in\Lambda-\{1\}$; $L_{6,8}^a$ and $L_{6,17}^a$, which are indexed by $a\in\Lambda$. Thus, if we denote by $r$ the cardinality of $\Lambda$, we obtain $28+2(r-1)+2r=26+4r$ minimal algebras.

If $\bk$ is algebraically closed (e.g. $\bk=\C$), then there are $30$ minimal
models over $\bk$. We can assume $a=1$ in lines $L_{6,8}^a$ and $L_{6,17}^a$, while lines $L_{6,2}^a$ and $L_{6,12}^a$ disappear (actually, they are equivalent to lines $L_3\oplus L_3$ and $L_{10}$ respectively).

Notice that when we set $a=0$, the minimal algebra $L_{6,2}^a$ reduces to $L_{6,1}$; the minimal algebra $L_{6,8}^a$
reduces to $L_{6,6}$; the minimal algebra $L_{6,12}^a$ reduces to $L_{6,9}$; and the minimal algebra
$L_{6,17}^a$ reduces to $L_{6,16}$.

\medskip

Finally, recall that this classification yields the classification of nilpotent Lie algebras of
dimension $6$ over $\bk$.

\section{$\bk$-homotopy types of $6$-dimensional nilmanifolds}

In the case $\bk=\Q$, the classification in Table 1 gives all rational homotopy types of $6$-dimensional
nilmanifolds. Note that $\Q^*/(\Q^* )^2$ is indexed by rational numbers up to squares, hence by
$a=\pm p_1p_2\ldots p_k$, where $p_i$ are different primes, and $k\geq 0$.

Let us explicitly give the classification of real homotopy types of $6$-dimensional nilmanifolds.
Note that $\R^*/(\R^*)^2=\{\pm 1\}$. Therefore there are $34$ real homotopy types, and
we have Table 2.

\begin{table}[h!]\label{table:2}
\caption{Real homotopy types of $6$-dimensional nilmanifolds}
\begin{tabular}{|c|c|c|c|c|c|c|c|c|c|c|}
\hline
 $(f_i)$    & $dx_3$ & $dx_4$ & $dx_5$ & $dx_6$ & $\g$ & \cite{Cerezo} & $b_1$ & $b_2$ & $b_3$ & $\sum_ib_i$\\
\hline
(6,0)   & 0 & 0 & 0 & 0 & $A_6$ & $-$ & 6 & 15 & 20 & 64\\
\hline
(5,1)   & 0 & 0 & 0 & $x_1x_2$ & $L_3\oplus A_3$ & $-$ & 5 & 11 & 14 & 48\\
\hline
     & 0 & 0 & 0 & $x_1x_2+x_3x_4$ & $L_{5,1}\oplus A_1$ & $-$ & 5 & 9 & 10 & 40\\
\hline
(4,2)  & 0 & 0 & $x_1x_2$ & $x_1x_3$ & $L_{5,2}\oplus A_1$ & $-$ & 4 & 9 & 12 & 40\\
\hline
     & 0 & 0 & $x_1x_2$ & $x_3x_4$ & $L_3\oplus L_3$ & $-$ & 4 & 8 & 10 & 36\\
\hline
     & 0 & 0 & $x_1x_2$ & $x_1x_3+x_2x_4$ & $L_{6,1}$ & $\mathcal{N}_{6,24}$ & 4 & 8 & 10 & 36\\
\hline
     & 0 & 0 & $x_1x_3-x_2x_4$ & $x_1x_4+x_2x_3$ & $L_{6,2}$ & $\mathcal{N}_{6,23}$ & 4 & 8 & 10 & 36\\
\hline
(4,1,1)   & 0 & 0 & $x_1x_2$ & $x_1x_5$ & $L_4\oplus A_2$ & $-$ & 4 & 7 & 8 & 32\\
\hline
     & 0 & 0 & $x_1x_2$ & $x_1x_5+x_3x_4$ & $L_{6,3}$ & $\mathcal{N}_{6,22}$ & 4 & 6 & 6 & 28\\
\hline
     & 0 & 0 & $x_1x_2$ & $x_1x_5+x_2x_3$ & $L_{5,3}\oplus A_1$ & $-$ & 4 & 7 & 8 & 32\\
\hline
(3,3)     & 0 & $x_1x_2$ & $x_1x_3$ & $x_2x_3$ & $L_{6,4}$ & $\mathcal{N}_{6,21}$ & 3 & 8 & 12 & 36\\
\hline
(3,2,1)   & 0 & $x_1x_2$ & $x_1x_3$ & $x_1x_4$ & $L_{6,5}$ & $\mathcal{N}_{6,20}$ & 3 & 6 & 8 & 28\\
\hline
     & 0 & $x_1x_2$ & $x_1x_3$ & $x_2x_4$ & $L_{6,6}$ & $\mathcal{N}_{6,18}$ & 3 & 6 & 8 & 28\\
\hline
     & 0 & $x_1x_2$ & $x_1x_3$ & $x_1x_5+x_2x_4$ & $L_{6,7}$ & $\mathcal{N}_{6,17}$ & 3 & 5 & 6 & 24\\
\hline
     & 0 & $x_1x_2$ & $x_1x_3$ & $x_2x_4+x_3x_5$ & $L_{6,8}^+$ & $\mathcal{N}_{6,15}$ & 3 & 5 & 6 & 24\\
\hline
     & 0 & $x_1x_2$ & $x_1x_3$ & $x_2x_4-x_3x_5$ & $L_{6,8}^-$ & $\mathcal{N}_{6,16}$ & 3 & 5 & 6 & 24\\
\hline
     & 0 & $x_1x_2$ & $x_1x_3$ & $x_1x_4+x_2x_3$ & $L_{6,9}$ & $\mathcal{N}_{6,19}$ & 3 & 6 & 8 & 28\\
\hline
(3,1,2)  & 0 & $x_1x_2$ & $x_1x_4$ & $x_2x_4$ & $L_{5,5}\oplus A_1$ & $-$ & 3 & 5 & 6 & 24\\
\hline
     & 0 & $x_1x_2$ & $x_1x_4$ & $x_2x_3+x_2x_4$ & $L_{6,10}$ & $\mathcal{N}_{6,12}$ & 3 & 5 & 6 & 24\\
\hline
     & 0 & $x_1x_2$ & $x_1x_4$ & $x_1x_3+x_2x_4$ & $L_{6,11}$ & $\mathcal{N}_{6,13}$ & 3 & 5 & 6 & 24\\
\hline
     & 0 & $x_1x_2$ & $x_1x_4+x_2x_3$ & $x_1x_3-x_2x_4$ & $L_{6,12}$ & $\mathcal{N}_{6,14}$ & 3 & 5 & 6 & 24\\
\hline
(3,1,1,1) & 0 & $x_1x_2$ & $x_1x_4$ & $x_1x_5$ & $L_{5,4}\oplus A_1$ & $-$ & 3 & 5 & 6 & 24\\
\hline
     & 0 & $x_1x_2$ & $x_1x_4$ & $x_1x_5+x_2x_3$ & $L_{6,13}$ & $\mathcal{N}_{6,11}$ & 3 & 5 & 6 & 24\\
\hline
     & 0 & $x_1x_2$ & $x_1x_4$ & $x_1x_5+x_2x_4$ & $L_{5,6}\oplus A_1$ & $-$ & 3 & 5 & 6 & 24\\
\hline
     & 0 & $x_1x_2$ & $x_1x_4$ & $x_1x_5+x_2x_3+x_2x_4$ & $L_{6,14}$ & $\mathcal{N}_{6,10}$ & 3 & 5 & 6 & 24\\
\hline
     & 0 & $x_1x_2$ & $x_1x_4+x_2x_3$ & $x_1x_5-x_3x_4$ & $L_{6,15}$ & $\mathcal{N}_{6,9}$ & 3 & 4 & 4 & 20\\
\hline
(2,1,2,1)   & $x_1x_2$ & $x_1x_3$ & $x_2x_3$ & $x_1x_4$ & $L_{6,16}$ & $\mathcal{N}_{6,8}$ & 2 & 4 & 6 & 20\\
\hline
     & $x_1x_2$ & $x_1x_3$ & $x_2x_3$ & $x_1x_4+x_2x_5$ & $L_{6,17}^+$ & $\mathcal{N}_{6,6}$ & 2 & 4 & 6 & 20\\
\hline
     & $x_1x_2$ & $x_1x_3$ & $x_2x_3$ & $x_1x_4-x_2x_5$ & $L_{6,17}^-$ & $\mathcal{N}_{6,7}$ & 2 & 4 & 6 & 20\\
\hline
(2,1,1,1,1)    & $x_1x_2$ & $x_1x_3$ & $x_1x_4$ & $x_1x_5$ & $L_{6,18}$ & $\mathcal{N}_{6,5}$ & 2 & 3 & 4 & 16\\
\hline
     & $x_1x_2$ & $x_1x_3$ & $x_1x_4$ & $x_1x_5+x_2x_3$ & $L_{6,19}$ & $\mathcal{N}_{6,4}$ & 2 & 3 & 4 & 16\\
\hline
     & $x_1x_2$ & $x_1x_3$ & $x_1x_4$ & $x_2x_5-x_3x_4$ & $L_{6,20}$ & $\mathcal{N}_{6,2}$ & 2 & 2 & 2 & 12\\
\hline
     & $x_1x_2$ & $x_1x_3$ & $x_1x_4+x_2x_3$ & $x_1x_5+x_2x_4$ & $L_{6,21}$ & $\mathcal{N}_{6,3}$ & 2 & 3 & 4 & 16\\
\hline
     & $x_1x_2$ & $x_1x_3$ & $x_1x_4+x_2x_3$ & $x_2x_5-x_3x_4$ & $L_{6,22}$ & $\mathcal{N}_{6,1}$ & 2 & 2 & 2 & 12\\
\hline
\end{tabular}
\end{table}

Notice that all these minimal algebras do actually correspond to nilmanifolds, since they
are defined over $\Q$.

The fifth column is a labeling of the nilpotent Lie algebra corresponding to the associated minimal algebra;
for instance, when we write $L_{5,1}\oplus A_1$ we mean that the
$6$-dimensional nilpotent Lie algebra splits as the sum
of a $5$-dimensional nilpotent Lie algebra with an abelian Lie algebra of
dimension $1$. In geometric terms, the corresponding $6$-dimensional nilmanifold
is the product of the corresponding $5$-dimensional nilmanifold with $S^1$.

The sixth column refers to the list contained in \cite{Cerezo}.
In \cite{Cerezo}, the problem of classifying $6$-dimensional nilmanifolds is treated in a
different way. Cerezo classifies $6$-dimensional nilpotent Lie algebras over $\R$.
Let us explain how we derived the correspondence between our list and his.
Consider, for example, the nilmanifold with real minimal model associated
to the Lie algebra $L_{6,14}$. The $6$-dimensional nilpotent Lie algebra $\N_{6,10}$ considered
by Cerezo has generators $\langle X_1,\ldots,X_6\rangle$ and commutators
 $$
 [X_1,X_2]=X_4, \ [X_1,X_4]=X_5, \  [X_1,X_5]=X_6, \ [X_2,X_3]=X_6 \quad \text{and} \quad [X_2,X_4]=X_6.
 $$
Using the correspondence between nilpotent Lie algebras and minimal algebras, according to formula
(\ref{eq:minimal-alg}), we associate the Lie algebra $\N_{6,10}$ to the nilmanifold $L_{6,14}$. To check the other correspondences, it might be necessary to switch variables.

The last columns contain the Betti numbers of the nilmanifolds, and the total dimension of the cohomology.
The computation of the Betti numbers has been perfomed using the following facts:
\begin{itemize}
\item Thanks to Poincar\'{e} duality, we have $b_0=b_6$, $b_1=b_5$ and $b_2=b_4$, where $b_i=\dim H^i(N)$.
\item Nilmanifolds are parallelizable and parallelizable manifolds have Euler characteristic zero, so
 \begin{equation}\label{eq:225}
 \sum_{i=0}^n(-1)^i\, b_i=0 \, .
 \end{equation}
\item to compute $b_3$ we use Poincar\'{e} duality and (\ref{eq:225}); we obtain
 \begin{equation}\label{eq:226}
 b_3=2(b_0-b_1+b_2).
 \end{equation}
\item $b_0=1$ and $b_1=f_1$.
\end{itemize}
Thus it is enough to compute $b_2$ to obtain the whole information.
As an example, we compute the Betti numbers of the nilmanifold $N=L_{6,12}$.
We have $b_0=b_6=1$ and $b_1=b_5=f_1=3$.
The computation of $b_2$ goes as follows: a basis for $\ker d \cap \wedge^2 V$ is given by
 $$
 \langle x_1x_2,x_1x_3,x_1x_4,x_1x_5+x_2x_6,x_1x_6-x_2x_5,x_2x_3,x_2x_4,x_3x_4+x_2x_6\rangle\, ,
 $$
and $\ker d \cap \wedge^2 V$ is $8$-dimensional. On the other hand, $\dim (\im \, d \cap \wedge^2V)= n-f_1= 3$.
Thus $b_2=\dim H^2(N)= 8-3=5=b_4$. This gives, according to (\ref{eq:226}), $b_3=6$ and $\sum_ib_i=24$.

Note that $\min \dim H^*(N) =12$. This agrees with \cite{Lupton}, proposition 3.3.

We end up with the proof of Theorem \ref{thm:main2}.

\noindent \emph{Proof of Theorem \ref{thm:main2}.} If $(\wedge V,d)$ is a  minimal model
of a nilmanifold, then it is defined over $\Q$. So it is a minimal algebra in Table 1,
with the condition that $a\in \Q^*$ if we are dealing with any of the four cases with parameter.
(This element $a$ is an invariant of the minimal algebra.)

Now, two nilmanifolds with minimal models $(\wedge V_1,d)$, $(\wedge V_2,d)$ are of the
same $\bk$-homotopy type if
$(\wedge V_1\ox \bk,d)$ and $(\wedge V_2 \ox \bk,d)$ are isomorphic (over $\bk$). Then, first they should be
in the same line in Table 1;
second, if they correspond to a parameter case, with respective parameters
$a_1, a_2\in \Q^*$, then the $\bk$-minimal models are isomorphic
if and only if there exists $\lambda \in \bk^*$
with $a_1=\lambda^2 a_2$. Therefore $a_1,a_2$ define the same class in $\Q^* /((\bk^*)^2\cap \Q^*)$. \fine

\begin{rem}
 A consequence of Theorem \ref{thm:main2} is that:
  \begin{enumerate}
  \item There are nilmanifolds which have the same real homotopy type but different rational homotopy type.
  \item There are nilmanifolds which have the same complex homotopy type but different real homotopy type.
  \item There are nilmanifolds $M_1,M_2$ for which the CDGAs $(\Omega^*(M_1),d)$ and $(\Omega^*(M_2),d)$
   are joined by chains of quasi-isomorphisms (i.e., they have the same \emph{real} minimal model),
   but for which there is no $f:M_1\to M_2$ inducing a quasi-isomorphism $f^* :(\Omega^*(M_2),d) \to
   (\Omega^*(M_1),d)$.
   Just consider $M_1,M_2$ not of the same rational homotopy type. If there was such $f$,
   then there is a map on the rational minimal models $f^*:(\wedge V_2,d) \to (\wedge V_1,d)$ such that
     $f^*_\R:(\wedge V_2\ox \R,d) \to (\wedge V_1\ox \R,d)$ is an isomorphism. Hence $f^*$ is an isomorphism
     itself, and $M_1, M_2$ would be of the same rational homotopy type.
    \end{enumerate}
\end{rem}

\begin{rem}
The fact that there exist nilpotent Lie algebras that are isomorphic over $\R$ but not over $\Q$ was noticed already by Lehmann in \cite{Lehmann}. He gave a particular example of two nilpotent $6$-dimensional Lie algebras that
are isomorphic over $\R$ but not over $\Q$.
\end{rem}

\section{Symplectic nilmanifolds}\label{sec:symplectic}

In this section we study which of the above rational homotopy types of nilmanifolds admit a symplectic structure. The subject is important because symplectic nilmanifolds which are not a torus supply a large source of examples of symplectic non-K\"{a}hler manifolds (see for instance \cite{Oprea-Tralle}).

In the $2$-dimensional case we have only the torus $\T^2$ which carries the symplectic area form $\omega=x_1x_2$.

The three $4$-dimensional examples are symplectic. We recall them:
\begin{enumerate}
 \item $dx_i=0$ for $i=1,2,3,4$. Here a symplectic for is given, for instance, by $\omega=x_1x_2+x_3x_4$;
 \item $dx_i=0$ for $i=1,2,3$ and $dx_4=x_1x_2$. Here we can take for example $\omega=x_1x_3+x_2x_4$;
 \item $dx_i=0$ for $i=1,2$, $dx_3=x_1x_2$ and $dx_4=x_1x_3$. Take $\omega=x_1x_4+x_2x_3$.
\end{enumerate}

In the $6$-dimensional case our approach is based on the following simple remark:
if there is a symplectic form, then there is an invariant symplectic form. Let $\omega
\in \wedge^2(x_1,\ldots,x_6)$. We can assume that it has rational coefficients, i.e.
 \begin{equation}\label{eq:551}
 \omega=\sum_{i<j} a_{ij}x_ix_j, \quad a_{ij}\in\Q.
 \end{equation}
In order for it to be a symplectic form, $\omega$ must be closed ($d\omega=0$) and
non-degenerate ($\omega^3\neq 0$). The second condition implies that
$\omega$ must be of the form
 \begin{equation}\label{eqn:552}
 \omega=a_{i_1i_2}x_{i_1}x_{i_2}+a_{i_3i_4}x_{i_3}x_{i_4}+a_{i_5i_6}x_{i_5}x_{i_6}+\omega'\, ,
 \end{equation}
where $i_1,\ldots,i_6$ is a permutation of $1,2,3,4,5,6$. If this is not possible then there is no symplectic
form $\omega$ and hence no symplectic structure on the associated nilmanifold.
We list the symplectic $6$-dimensional nilmanifolds in Table 3. In the first column we mention the
Lie algebra of Table 2 associated to the rational homotopy type of the nilmanifold. In the second column either
we produce an explicit symplectic form for the type, or we say that there does not exist
symplectic structures on it.

\begin{table}[h!]\label{table:3}
\caption{Symplectic $6$-dimensional nilmanifolds}
\begin{tabular}{c|c||c|c}
\toprule[0.04 cm]
\textbf{Type} & \textbf{Symplectic form} & \textbf{Type} & \textbf{Symplectic form}\\
\midrule
$A_6$ & $x_1x_2+x_3x_4+x_5x_6$ & $L_{5,5}\oplus A_1$ & Not symplectic\\
\midrule
$L_3\oplus A_3$ & $x_1x_6+x_2x_3+x_4x_5$ & $L_{6,10}$ & $x_1x_6+x_2x_5-x_3x_4$\\
\midrule
$L_{5,1}\oplus A_1$ & Not symplectic & $L_{6,11}$ & $x_1x_5+x_2x_6+x_3x_4$\\
\midrule
$L_{5,2}\oplus A_1$ & $x_1x_5+x_2x_4+x_3x_6$ & $L_{6,12}$ & $x_1x_6+2x_2x_5+x_3x_4$\\
\midrule
$L_3\oplus L_3$ & $x_1x_5+x_3x_6+x_2x_4$ & $L_{5,4}\oplus A_1$ & $x_1x_3+x_2x_6-x_4x_5$\\
\midrule
$L_{6,1}$ & $x_1x_3+x_2x_6+x_3x_5$ & $L_{6,13}$ & $x_1x_3+x_2x_6-x_4x_5$\\
\midrule
$L_{6,2}$ & $x_1x_6+x_2x_5+x_3x_4$ & $L_{5,6}\oplus A_1$ & $x_1x_3+x_2x_6-x_4x_5$\\
\midrule
$L_4\oplus A_2$ & $x_1x_6+x_2x_5+x_3x_4$ & $L_{6,14}$ & $x_1x_3+x_2x_6-x_4x_5$\\
\midrule
$L_{5,3}\oplus W$ & $x_1x_6+x_2x_4-x_3x_5$ & $L_{6,15}$ & $x_1x_4+x_2x_6+x_3x_5$\\
\midrule
$L_{6,3}$ & Not symplectic & $L_{6,16}$ & $x_1x_6+x_1x_5+x_2x_4+x_3x_5$\\
\midrule
$L_{6,4}$ & $x_1x_4+x_2x_6+x_3x_5$ & $L_{6,17}^+$ & $x_1x_6+x_1x_5+x_2x_4+x_3x_5$\\
\midrule
$L_{6,5}$ & $x_1x_6+x_2x_4+x_3x_5$ & $L_{6,17}^-$ & $x_1x_6+x_1x_5+x_2x_4+x_3x_5$\\

\midrule
$L_{6,6}$ & $x_1x_4+x_2x_6+x_3x_5$ & $L_{6,18}$ & $x_1x_6+x_2x_5-x_3x_4$\\
\midrule
$L_{6,7}$ & Not symplectic & $L_{6,19}$ & $x_1x_6+x_2x_4+x_2x_5-x_3x_4$\\
\midrule
$L_{6,8}^+$ & Not symplectic & $L_{6,20}$ & Not symplectic\\
\midrule
$L_{6,8}^-$ & Not symplectic & $L_{6,21}$ & $2x_1x_6+x_2x_5+x_3x_4$\\
\midrule
$L_{6,9}$ & $x_1x_6+2x_2x_5+x_3x_4$ & $L_{6,22}$ & Not symplectic\\
\bottomrule[0.04 cm]
\end{tabular}
\end{table}

As an example of computations, we show that the nilmanifold $L_{5,5}\oplus A_1$
is not symplectic and also how we constructed one possible symplectic form on $L_{6,9}$.
The minimal model of $L_{5,5}\oplus A_1$ is $(\wedge V,d)$ with
$$
dx_4=x_1x_2, \ dx_5=x_1x_4 \  \ \mathrm{and} \ dx_6=x_2x_4.
$$
It is easy to see that the space of closed elements of degree $2$ is generated by
 $$
 x_1x_2, x_1x_3, x_2x_3, x_1x_4, x_2x_4, x_1x_5,x_2x_5+x_1x_6, x_2x_6 \, ,
$$
so $\omega$ is a linear combination of these terms. But now, according to (\ref{eqn:552}),
the subindices $5,6$ do not go together, and $5$ goes either with $1$ or $2$, whereas
$6$ goes either with $1$ or $2$. This implies that $3,4$ should form a pair, which it is
impossible.

To show that some nilmanifold admits some symplectic structure is much easier: it is enough to find a symplectic form. If we take $L_{6,9}$ we have the minimal model $(\wedge V,d)$ with the following differentials:
 $$
 dx_4=x_1x_2, \ dx_5=x_1x_3 \ \ \mathrm{and} \ dx_6=x_1x_4+x_2x_3.
 $$
Now $d(x_1x_6)=d(x_3x_4)=-x_1x_2x_3$ and $d(x_2x_5)=x_1x_2x_3$. Therefore
 $$
 \omega=x_1x_6+2x_2x_5+x_3x_4
 $$
is closed and we easily see that $\omega^3=12\, x_1x_2x_3x_4x_5x_6\neq 0$. Thus $\omega$ is symplectic.

\section*{Appendix}

This appendix is devoted to the study
of the minimal model of commutative differential graded algebras defined over fields of characteristic $p\neq 2$.
Let $\bk$ be a field of arbitrary characteristic $p\neq 2$.

\begin{teo}
 Any CDGA $(A,d)$ has a Sullivan model: there exist a minimal algebra $(\wedge V,d)$ (in the sense of the definition
 given in the introduction) and a quasi-isomorphism $(\wedge V,d)\to (A,d)$.
\end{teo}

\begin{proof}
The proof of the existence is the same as in the case of characteristic zero, given in (\cite{FHT}, chapter 14).
\end{proof}

Now we want to study the issue of uniqueness of the minimal model. It is not known in general whether if
$(\wedge V,d)\to (A,d)$ and $(\wedge W,d)\to (A,d)$ are two minimal models, then
$(\wedge W,d)\cong (\wedge V,d)$ necessarily. This is known in characteristic zero (\cite{S}),
but it is an open question in positive characteristic $p\neq 2$ (see \cite{Ha}).

Here we give a positive answer for the case of CDGAs with a minimal model generated in degree $1$.
However, some of the results which follow are valid in full generality.

\begin{lemma}\label{lemma2}
Let $(\wedge V,d)$ be a minimal algebra and let $(A,d)$ and $(B,d)$ be two CDGAs. Suppose that $f:(\wedge V,d)\to(A,d)$ is a CDGA morphism and that $\pi:(B,d)\to(A,d)$ is a surjective quasi-isomorphism. Then $f$ can be lifted to a CDGA map $g:(\wedge V,d)\to(B,d)$ such that the following diagram is commutative:
$$
\xymatrix
{
& (B,d)\ar[d]^{\pi}\\
(\wedge V,d)\ar[r]^f\ar@{-->}[ur]^g & (A,d)}
$$
Moreover, if $f$ is a quasi-isomorphism, then so is $g$.
\end{lemma}

\begin{proof}
We work inductively. By minimality, there is an increasing filtration $\{V_\mu\}$
of $V$ such that $d$ maps $V_\mu$ to $\wedge (V_{<\mu})$ ($V_\mu$ is the span of
those generators $x_\tau$ with $\tau\leq \mu$).
Suppose that $g$ has been constructed on $V_{<\mu}$ and consider $x=x_\mu$. Since
$dx\in \wedge (V_{<\mu})$, $g(dx)$ is well defined. We want to solve
  \begin{equation}\label{eqn:ap1}
  \left\{ \begin{array}{l}
           g(dx)=dy \\ f(x)=\pi(y),
          \end{array}
  \right.
  \end{equation}
so that we can set $g(x)=y$.

There is some $b\in B$ such that $\pi(b)=f(x)$. Then
$\pi( g(dx))=f(dx)=d(f(x))=d(\pi(b))=\pi(db)$, so $c=g(dx)-db\in \ker \pi$.
We compute $dc=d(g(dx))=0$, so $c$ is closed.
But $[c] \in H^*(B)\cong H^*(A)$ and $\pi(c)=0$, so $[c]=0$, i.e. there is
some $e\in B$ such that $c=de$. Now $d\pi(e)=\pi(c)=0$, so $\pi(e)$ is closed
and $[\pi(e)]\in H^*(A)\cong H^*(B)$. Hence there is some closed $\beta\in B$ and
$\alpha \in A$ such that $\pi(e)=\pi(\beta)+d\alpha$. Using the surjectivity of
$\pi$ again, $\alpha=\pi(\psi)$, for some $\psi\in B$. So $\pi(e)=\pi(\beta+d\psi)$.
Now take $y=b+e-\beta-d\psi$. Clearly $\pi(y)=\pi(b)=f(x)$ and $dy=db+de=g(dx)$.\\
Now suppose that $f$ is a quasi-isomorphism and denote $f_*$ and $\pi_*$ the maps induced by $f$ and $\pi$ respectively at cohomology level. One has $f=\pi\circ g$, hence $f_*=\pi_*\circ g_*$; thus $g_*=\pi_*^{-1}\circ f_*$ is also an isomorphism.
\end{proof}

Now we particularise to minimal algebras generated in degree $1$. In this case,
we do not need surjectivity to prove a lifting property.

\begin{teo} \label{thm:apx}
Let $(\wedge V,d)$ be a minimal algebra  generated in degree
 $1$ (i.e. $V=V^1$), and let $(A,d)$ and $(B,d)$ be two CDGAs. Suppose that $A^0=\bk$.
If $f:(\wedge V,d)\to(A,d)$ is a CDGA morphism and
$\psi:(B,d)\to(A,d)$ is a quasi-isomorphism, then there exists a CDGA map
$g:(\wedge V,d)\to(B,d)$ such that $\psi\circ g=f$.

Moreover, if $f$ is a quasi-isomorphism, then so is $g$.
\end{teo}

\begin{proof}
  We work as in the proof of lemma \ref{lemma2}. Consider generators $\{x_\tau\}$ of
$V=V^1$. Assume that $g$ has been defined for $V_{<\mu}$, and let $x=x_\mu$. Since
$dx\in \wedge^2 (V_{<\mu})$, $g(dx)$ is well defined. As before, we want to solve
(\ref{eqn:ap1}).

Now $d(g(dx))=g(dd(x))=0$, so $[g(dx)]\in H^2(B,d)$. But
$\psi_*[g(dx)]=[\psi( g(dx))]=[f(dx)]=[d(f(x))]=0$, so $[g(dx)]=0$. Therefore, there
exists $\xi\in B^1$ such that $g(dx)=d\xi$. Now $d(\psi(\xi))= \psi(g(dx))=f(dx)=d(f(x))$,
so $\psi(\xi)-f(x)\in A^1$ is closed.
As $A^0=\bk$, we have that $H^1(A,d)=Z^1(A,d)=\ker (d:A^1\to A^2)$. Clearly the quasi-isomorphism $\psi:
(B,d)\to (A,d)$ gives a surjective map $Z^1(B,d)\to Z^1(A,d)$. Therefore,
there exists $b\in Z^1(B,d)\subset B^1$ such that $\psi(\xi)-f(x)=\psi(b)$.
Take $y=\xi-b$, to solve (\ref{eqn:ap1}).

\end{proof}




\begin{lemma}\label{lemma3}
Suppose $\varphi:(\wedge V,d)\to (\wedge W,d)$ is a quasi-isomorphism between minimal algebras.
Then $\varphi$ is an isomorphism.
\end{lemma}

\begin{proof}
We can assume inductively that $\wedge (V^{< n})\cong \wedge (W^{< n})$. 
We first show that $\varphi:\wedge (V^{\leq n})\to \wedge (W^{\leq n})$ is injective.
It is enough to see that the composition $\bar\f: V^n\to (\wedge W^{\leq n})^n \to W^n$
is injective. Suppose $v\in V^n$ satisfies $\bar\f(v)=0$. Then there exists $v'
\in \wedge (W^{< n})\cong\wedge (V^{< n})$ such that $\f(v)=\f(v')$. Then $\f(v'')=0$,
where $v''=v-v'$.
Then
 $$
 0=d(\varphi(v''))=\varphi(dv'').
 $$
 Thus $dv''=0$. 
 Since $\varphi$ is a quasi-isomorphism and $\varphi^*[v'']=0$, we have that $v''=d(v''')$
 for some $v'''\in(\wedge V)^{n-1}$; but this is impossible since $\wedge V$ is a
 minimal algebra.

 Now we prove the surjectivity of $\varphi:\wedge (V^{\leq n})\to \wedge (W^{\leq n})$.
 First note that the minimality condition means the existence of an increasing
 filtration $V^n_i$  such that $d(V^n_i) \subset \wedge
 (V^{<n} \oplus V^n_{i-1})$ (and an analogous filtration $W^n_i$ for $W^n$).
 We assume by induction that $\wedge(V^{<n} \oplus V^n_{i-1}) \cong
 \wedge(W^{<n} \oplus W^n_{i-1})$. Consider
 $$
 \cV_i = V^n_i \oplus \wedge (V^{<n} \oplus V^n_{i-1}).
 $$
 These are differential vector
subspaces. Write $\cV_i \inc \wedge V \to C$, where $C$ is the cokernel. Then $C$ has
only terms of degree $\geq n$. Moreover if we take the filtration with $V^n_i$ maximal (i.e. $\cV_i=d^{-1}(\wedge(V^{<n}\oplus V^n_{<i}$)), then $H^n(C)=0$. This implies that $H^{\leq n}(\cV_i)\cong H^{\leq n}(\wedge V)$ and
$H^{n+1}(\cV_i)\inc H^{n+1}(\wedge V)$.

We define analogously $\cW_i = W^n_i \oplus \wedge (W^{<n} \oplus W^n_{i-1})$. Clearly
$\varphi:\cV_i \to \cW_i$.
We have an exact sequence $0\to \cV_i \to \cW_i \to Q \to
0$, where $Q=W^n_i/V_i^n$ is the cokernel. Again, $Q$ does not have terms of degree $<n$. Also
$d$ on $Q^n$ is zero, so $H^n(Q)=Q^n$.
Note that the isomorphism $H^*(\wedge V)\cong
H^*(\wedge W)$ implies that $H^{\leq n}(\cV_i)\cong H^{\leq n}(\cW_i)$ and
$H^{n+1}(\cV_i)\inc H^{n+1}(\cW_i)$. The long exact sequence in cohomology gives
$H^n(Q)=Q^n=0$, and hence $\cV_i \cong \cW_i$, which completes the induction.
\end{proof}

This gives us the uniqueness of the minimal model for the CDGAs that we are interested in.

\begin{teo} \label{cor:ap3}
 Let $(A,d)$ be a CDGA, defined over a field $\bk$ of characteristic $p\neq 2$, such that $A^0=\bk$. Suppose that its
 minimal model $\varphi:(\wedge V,d)\to (A,d)$ satisfies that $(\wedge V,d)$ is a minimal algebra generated in degree
 $1$. If $(\wedge W,d)\to (A,d)$ is another minimal model for
 $(A,d)$, then $(\wedge W,d)\cong (\wedge V,d)$.
\end{teo}

\begin{proof}
 By Theorem \ref{thm:apx}, there exists a quasi-isomorphism $g:(\wedge V,d)\to (\wedge W,d)$. By Lemma \ref{lemma3},
$g$ is an isomorphism.
\end{proof}

We have the following refinement.

\begin{cor}
Consider the category of CDGAs \ $(A,d)$ with $A^0=\bk$ and whose minimal model is
generated in degree $1$.
If two of such CDGAs \ $(A,d)$ and $(B,d)$ are quasi-isomorphic, then they have the same minimal model.
\end{cor}

\begin{proof}
Without loss of generality, we may assume that there is a quasi-isomorphism $\psi:(B,d)\to (A,d)$.
If $\varphi:(\wedge V,d)\to (A,d)$ is a minimal model for $(A,d)$ then there exists
a quasi-isomorphism $g:(\wedge V,d)\to (B,d)$. Any other minimal model of $(B,d)$ is isomorphic to
$(\wedge V,d)$ by Theorem \ref{cor:ap3}.
\end{proof}



\begin{thebibliography}{22}

 \bibitem{Cerezo} \textsc{Cerezo, A.}, \textit{Les alg\`{e}bres de Lie nilpotentes r\'{e}elles et complexes de
 dimension $6$}, Prepublication J. Dieudonn\'{e}, Univ. de Nice \textbf{27}, 1983.

 \bibitem{DGMS} \textsc{Deligne, P., Griffiths P., Morgan J. and Sullivan D.}, \emph{Real Homotopy Theory of
 K\"{a}hler Manifolds}, Inventiones Mathematic\ae{} Vol. 29, 1975, p. 245-274.

 \bibitem{FHT} \textsc{F\'{e}lix, Y., Halperin, S. and Thomas, J.C.}, \textit{Rational Homotopy Theory}, Graduate
 Texts in Mathematics \textbf{205}, Springer, 2001.

 \bibitem{Goze} \textsc{Goze, M. and Khakimdjanov, Y.}, \emph{Nilpotent Lie algebras}, Mathematics and
 its Applications \textbf{361}, Kluwer, 1996.

 \bibitem{de-Graaf} \textsc{de Graaf, W. A.} \textit{Classification of $6$-dimensional nilpotent Lie algebras over
 fields of characteristic not $2$}, Journal of Algebra \textbf{309}, 2007, 640-653.

 \bibitem{GM} \textsc{Griffiths, P. and Morgan, J.} \textit{Rational Homotopy Theory and Differential Forms},
 Progress in Mathematics, Birkh\"{a}user, 1981


 \bibitem{Ha} \textsc{Halperin, S.} \emph{Universal enveloping algebras and loop space homology},
 Journal of Pure and Applied Algebra \textbf{83}, 1992, 237-282.

 \bibitem{Lehmann} \textsc{Lehmann, D.}, \emph{Sur la g\'en\'eralisation d'un th\'eor\`eme de Tischler \`a certains
feuilletages nilpotents},  Nederl. Akad. Wetensch. Indag. Math. \textbf{41}, 1979, 177-189.

 \bibitem{Lupton} \textsc{Lupton, G.}, \emph{The Rational Toomer Invariant and Certain Elliptic Spaces},
 Contemporary Mathematics \textbf{316}, 2004, 135-146.

 \bibitem{Magnin} \textsc{Magnin, L.}, \emph{Sur les alg\`ebres de Lie nilpotentes de dimension $\leq 7$},
 Journal of Geometry and Physics \textbf{3}, 1986, 119-144.


 \bibitem{Morgan} \textsc{Morgan, J.}, \emph{The Algebraic Topology of Smooth Algebraic Varieties},
 Publications math\'{e}matiques de l'I.H.\'{E}.S \textbf{48}, 1978, 137-204.

 \bibitem{Nomizu} \textsc{Nomizu, K.}, \emph{On the cohomology of compact homogeneous spaces of nilpotent Lie
 groups}, Annals of Mathematics (2) \textbf{59}, 1954, 531-538.

 \bibitem{Oprea-Tralle} \textsc{Oprea, J. and Tralle, A.}, \emph{Symplectic Manifolds with no K\"{a}hler Structure},
 Lecture Notes in Mathematics \textbf{1661}, Springer, 1997.

 \bibitem{Salamon} \textsc{Salamon, S.} \emph{Complex structures on nilpotent Lie algebras},
 Journal of Pure and Applied Algebra \textbf{157}, 2001, 311-333.

\bibitem{S} \textsc{Sullivan, D.}, \emph{Infinitesimal Computations in Topology},
 Publications math\'{e}matiques de l'I.H.\'{E}.S. \textbf{47}, 1997, 269-331.





\end{thebibliography}
\end{document}